\documentclass[12pt,a4paper]{article}
\usepackage{amsmath}
\usepackage{amssymb}
\usepackage{graphicx}
\usepackage{tikz}
\usetikzlibrary{calc}
\usepackage{xcolor}
\usepackage{esint}
\usepackage{amsthm}

\theoremstyle{plain}
\begingroup
\newtheorem{theorem}{Theorem}[section]
\newtheorem{lemma}[theorem]{Lemma}
\newtheorem{proposition}[theorem]{Proposition}
\newtheorem{corollary}[theorem]{Corollary}
\endgroup

\theoremstyle{definition}
\begingroup
\newtheorem{definition}[theorem]{Definition}
\newtheorem{remark}[theorem]{Remark}

\endgroup

\title{A homogenization result for weak membrane energies }
\author{ Leonard Kreutz\\
University of Vienna, Oskar-Morgenstern-Platz 1,\\ 1090 Vienna, Austria\\ {\tt e-mail: leonard.kreutz@gmail.com}
}
\date{}  

\usepackage[english]{babel}

\begin{document}
\maketitle

\begin{abstract} We prove a $\Gamma$-convergence result for space dependent weak membrane energies, that is for 'truncated quadratic potentials', that are quadratic below some threshold (depending on the pair of points that we are considering) and constant above. We prove that the limit surface energy density coincides with the one for spin systems, while the bulk energy density is not affected by the different levels of truncation and coincides with an purely elastic energy density.
\end{abstract}


\medskip


\section{Introduction}
In recent years a number of variational models related to reconstruction problems in Computer Vision have been proposed (for a survey see e.g. the monographs \cite{aubert2006mathematical,blake1987visual,morel1996variational}). For the image segmentation problem, Mumford and Shah \cite{mumford1989optimal} proposed to minimize the functional
\begin{align}\label{MSFunctional}
F(u,C) = \int_{\Omega \setminus C} |\nabla u|^2\mathrm{d}x + c_1\int_\Omega|u-g|^2\mathrm{d}x + c_2\mathcal{H}^1(C),
\end{align}
where $\Omega \subset \mathbb{R}^2$ is a bounded open set (the image domain), $\mathcal{H}^1$ denotes the one-dimensional Hausdorff measure, $g \in L^\infty(\Omega)$ is the input image and $c_1,c_2>0$ are tuning parameters. The functional is minimized over all closed sets $C \subset \overline{\Omega}$ and all $u \in C^1(\Omega\setminus C)$. To obtain existence of (\ref{MSFunctional}) it is convenient to rewrite it in a weaker form, as
\begin{align}\label{MSSBV}
F(u) = \int_{\Omega \setminus S(u)} |\nabla u|^2\mathrm{d}x + c_1\int_\Omega|u-g|^2\mathrm{d}x + c_2\mathcal{H}^1(S(u)),
\end{align}
where $u \in SBV(\Omega)$ denotes the space of special functions of bounded variation on $\Omega$ introduced by De Giorgi and Ambrosio (see \cite{ambrosio1989variational},\cite{ambrosio2000functions}), and $S(u)$ denotes the measure theoretic discontinuity set of $u$. A key point is the lower-semicontinuity of (\ref{MSSBV}) with respect to the strong $L^1$-topology with respect to which the functional is coercive.
The function $u$ represents a denoised approximation of the input image, and $S(u)$ represents the boundaries of the segmentation. The discrete counterpart of this minimization problem is the minimization of the Blake Zisserman 'weak membrane' energie (see \cite{chambolle1995image}) given by
\begin{align*}
E(u) = \sum_{(i,j) \in \mathbb{Z}^2\cap \Omega} W\left(u_{i+1,j}-u_{i,j}\right) + W\left(u_{i,j+1}-u_{i,j}\right) + |u_{i,j}-g_{i,j}|^2,
\end{align*}
where $g=g_{i,j}$ it the input image, i.e. a grey level function defined on the matrix of pixels describing $\Omega$ and the potential $W : \mathbb{R} \to \mathbb{R}$ is given by
\begin{align*}
W(x) = \min\left\{\lambda z^2, \alpha\right\}
\end{align*}
is a truncated parabola. Here $\lambda,\alpha >0$ are tuning parameters that have to be adjusted to fit the model to a particular case (in the following we assume $\lambda,\alpha=1$). The particular shape of $W$ has a regularizing effect whenever the threshold $\left( u_{(i,j)+e_k}-u_{(i,j)}\right)^2 \leq 1$ is not exceeded. The two pixels  $u_{(i,j)+e_k}$ and $u_{(i,j)}$ should remain close unless their difference exceeds a certain threshold in which case the spring binding them is broken. The discrete jump set $S(u)$ can then be seen as the set of springs that are broken. Antonin Chambolle proved in  \cite{chambolle1995image} that energies, rescaled in a suitable way approximate well in the sense of $\Gamma$-convergence (see \cite{braides2006handbook,dal2012introduction}) an anistropic version of the (\ref{MSFunctional}). The rescaled energies are given by
\begin{align*}
E_\varepsilon(u) = \sum_{(i,j) \in \varepsilon\mathbb{Z}^2 \cap \Omega} \varepsilon^2\left( W_\varepsilon\left(\varepsilon^{-1}(u_{i+\varepsilon,j}-u_{i,j})\right) + W_\varepsilon\left(\varepsilon^{-1} (u_{i,j+\varepsilon}-u_{i,j})\right) + |u_{i,j}-g_{i,j}|^2\right),
\end{align*}
where \begin{align*}
W_\varepsilon(z) = \min\left\{ z^2, \varepsilon^{-1}\right\}.
\end{align*}
The limit functional is given by
\begin{align}\label{MSSBV}
F(u) = \int_{\Omega \setminus S(u)} |\nabla u|^2\mathrm{d}x + c_1\int_\Omega|u-g|^2\mathrm{d}x + c_2\int_{S(u)}||\nu_u(x)||_1\mathrm{d}\mathcal{H}^1, 
\end{align}
with $u \in SBV(\Omega)$, $\nu_u(x)$ denotes the measure theoretic normal to $S(u)$ at the point $x \in S(u)$ and $||\nu||_1 = |\nu_1|+|\nu_2|$. The result has further been generalized to higher dimensions and to long range interactions by Chambolle in \cite{chambolle1999finite} and by Braides and Gelli in \cite{braides1999limits}. Note however that the results stated in those articles take into account the same interactions at every point of the matrix.

\bigskip

 Homogenization of free discontinuity problems in the continuous setting has been discussed in \cite{braides1996homogenization,cagnetti2017gamma,giacomini2006gamma}.  A crucial point in the analysis is the separation of the surface and the elastic contribution. That is the admissible minimizers in the homogenization formula for the elastic and the surface energy density can be restricted to functions $u \in W^{1,2}$ or piecewise constant functions $u$ respectively. A crucial step in the demonstration of that is the usage of a Coarea formula in order to reduce from function which have vanishing elastic energy to functions which have zero elastic energy, i.e. piecewise constant functions.

 \bigskip

 The scope of this article is for a general class of finite range interaction potentials to derive the limiting theory with focus on the surface energy density and draw comparisons to the surface energies that one obtains by homogenizing spin systems, that is we consider energies of the type
\begin{align}\label{Ourenergies}
F_\varepsilon(u) = \sum_{i \in \varepsilon\mathbb{Z}^d\cap \Omega}\sum_{\xi \in V} \varepsilon^d W_\varepsilon^{i,\xi}(\varepsilon^{-1}(u_{i+\varepsilon\xi} -u_i)),
\end{align}
where $\Omega\subset \mathbb{R}^d$ denotes a bounded open regular set, $V$ is a finite subset of $ \mathbb{Z}^d$ containing the standard orthonormal basis $e_1,\ldots,e_d$, and $W_\varepsilon^{i,\xi} : \mathbb{R} \to \mathbb{R}$ is given by
\begin{align*}
W(z) = \min\left\{ z^2, \varepsilon^{-1}c_{i,\xi}^\varepsilon\right\}
\end{align*}
where $c_{i,\xi}^\varepsilon \geq 0$ and $\inf c_{i,e_k}^\varepsilon >0$, where the infimum is taken over $\varepsilon>0,k\in \{1,\ldots,d\}$ and $i \in \varepsilon\mathbb{Z}^d \cap \Omega$. (We omit the fidelity term, sine it is only a continuous perturbation of (\ref{Ourenergies})). We show that the $\Gamma$-limit (which exists up to subsequences) of (\ref{Ourenergies}) is given by
\begin{align*}
F(u) = \int_{\Omega\setminus S(u)} f(x,\nabla u)\mathrm{d}x + \int_{S(u)}\varphi(x,\nu_u(x))\mathrm{d}\mathcal{H}^{d-1},
\end{align*}
where  $f,\varphi : \Omega \times \mathbb{R}^d \to [0,+\infty]$ are characterized by asymptotic cell formulas. The surface energy density $\varphi$ is shown to agree with the surface energy density of $\Gamma$-limit of spin energies of the form
\begin{align}\label{Spinenergies}
E_\varepsilon(u)= \frac{1}{8}\sum_{i \in \varepsilon\mathbb{Z}^d\cap \Omega}\sum_{\xi \in V} \varepsilon^{d-1} c_{i,\xi}^\varepsilon (u_{i+\varepsilon\xi}-u_i)^2,
\end{align}
where $u_i \in \{\pm 1\}$, $c_{i,\xi}^\varepsilon$ as above. Integral representation formulas of the $\Gamma$-limit are presented in \cite{alicandro2016local,braides2013homogenization,caffarelli2005interfaces} where it is shown, that the energy density can be recovered by
\begin{align}
\varphi(x_0,\nu) = \lim_{\rho \to 0} \frac{1}{\rho^{d-1}}\lim_{\eta \to 0}\limsup_{\varepsilon \to 0} \inf \Bigg\{&E_\varepsilon(v, Q^\rho_\nu(x_0)) : v \in \mathcal{PC}_\varepsilon(\mathbb{R}^d;\{-1,+1\}) \\& \nonumber: v_i = (u_{x_0,\nu})_i \text{ for all } i \in  Z_\varepsilon\left(\left(Q_\nu^\rho(x_0)\right)_\eta \right)\Bigg\}.
\end{align}

\bigskip

Finally in the case of non-degeneracy of the interaction-coefficients, that is there exists $0<c<C<+\infty$ such that $c_{i,\xi}^\varepsilon \in [c,C] \cap \{0\}$ for all $\varepsilon>0,i \in \varepsilon\mathbb{Z}^d$ and $\xi \in V$, we perform a discrete Lusin type approximation of our piecewise constant functions to recover the bulk energy density $f$ as the energy density of purely elastic energies. The elastic energies are given by
\begin{align}\label{ElasticEnergies}
H_\varepsilon(u) = \sum_{i \in \varepsilon\mathbb{Z}^d\cap \Omega}\sum_{\xi \in V} \varepsilon^d \mathrm{1}_{\{c_{i,\xi}^\varepsilon > 0\}} (u_{i+\varepsilon\xi}-u_i)^2
\end{align}
whose $\Gamma$-limit (up to subsequences) is shown to exist in \cite{alicandro2004general} and takes the form
\begin{align*}
H(u) = \int_\Omega h(x,\nabla u)\mathrm{d}x.
\end{align*}
In section 5 we prove that 
\begin{align*}
h(x,\zeta) = f(x,\zeta)
\end{align*}
for almost all $x \in \Omega$ and all $\zeta \in \mathbb{R}^d$.

\bigskip

It is noteworthy that even if the form of the potentials considered as truncated parabolas seems particular their behaviour is in a sense universal that is they describe at least in an approximative sense more general convex-concave energies. The interested reader can check \cite{braides2006effective} for the relation between Lennard-Jones type potentials and the truncated parabolas in dimension 1 or \cite{braides2014global} for the (gradient flow) dynamical case. In \cite{braides2008asymptotic} it is explained how to rigorously to formulate their relationship using the technique of asymptotic expansion.

\bigskip

The paper is organized as follows. In section 2 we recall some notation and introduce the technical tools needed to perform the analysis. In section 3 we state the setting of the problem, recall some already known results and state the main theorem. In section 4 we perform the proof of the main theorem. Finally in section 5 we characterize the bulk energy density.

\section{Notation and Preliminaries}

In this chapter we introduce some notation and recall some results about the theory of functions of bounded variation as well as $\Gamma$-convergence.

\medskip

We assume that $\Omega \subset \mathbb{R}^d$ is a bounded open and Lipschitz set. We set $Q=(-\frac{1}{2},\frac{1}{2})^d$ the open unit cube with side length $1$ centred in $0$. For $\nu \in S^{d-1}$ we define $Q^\nu = R_\nu Q$, where $R_\nu$ is a rotation such that $R_\nu e_d = \nu$, where $e_1,\ldots,e_d$ stands for the canonical basis in $\mathbb{R}^d$. For a borel set $B \in \mathcal{B}(\Omega)$ we denote by $|B|$ the $d$-dimensional lebesgue measure of the set $B$. Finally we set $Q^\nu_\rho(x_0) = \rho Q^\nu + x_0$, where $\rho >0, \nu \in S^{d-1}$ and $x_0 \in \mathbb{R}^d$, we omit $\nu$ (resp. $\rho$) if $\nu=e_d$ (resp $\rho=1$). For the general theory of functions of bounded variation we refer to \cite{ambrosio2000functions,giusti1984minimal}.
Let $\Omega$ be an open bounded subset of $\mathbb{R}^d$. For $A \subset \mathbb{R}^d$ we define
$
A_\eta = \{x \in A :\mathrm{dist}(x,A^c) <\eta\}
$, $
A^+_\eta = \{x \in A :\mathrm{dist}(x,A^c) >\eta\}
$. We set
\begin{align*}
u_{x_0,\nu}^{z_1,z_2}(x) = \begin{cases}
z_2 &(x-x_0)\cdot \nu \geq 0\\
z_1 &\text{otherwise.}
\end{cases}
\end{align*}
We write $u_{x_0,\nu} = u_{x_0,\nu}^{-1,1}$.
For $u :\mathbb{R}^d \to \mathbb{R}$, $\varepsilon>0$ and $\xi \in \mathbb{Z}^d$ we define
\begin{align*}
D^\xi_\varepsilon u(x) = \frac{u(x+\varepsilon\xi)-u(x)}{\varepsilon}.
\end{align*}

 For $u \in L^1(\Omega)$ we define $u_T = (u\vee T)\wedge (-T)$. We say that $u \in L^1(\Omega)$ is a function of bounded variation if its distributional derivative $ Du \in [\mathcal{M}(\Omega)]^d $. 
We say that $u \in L^1(\Omega)$ is approximately continuous at $x \in \Omega$ if 
\begin{align*}
\lim_{\rho \to 0} \fint_{B_{\rho}(x)} |u(z)-u(x)|\mathrm{d}x =0.
\end{align*}
The set $S(u)$ of points where this property does not hold is called the approximately discontinuity set. If $u \in BV(\Omega)$, then $S(u)$ is $(d-1)$-rectifiable, i.e.
\begin{align}\label{rectifiability}
S(u) = N \cup \left(\bigcup_{i \in \mathbb{N}} \Gamma_i\right),
\end{align}
where $\mathcal{H}^{d-1}(N)=0$ and $\{\Gamma_i\}_i$ is a sequence of compact sets each contained in a $C^1$ hypersurface $\Gamma_i$. Moreover there exist borel functions $\nu_u : S(u) \to S^{d-1}$, $u^{\pm} : S(u) \to \mathbb{R}$ such that for $\mathcal{H}^{d-1}$ a.e. $z \in S(u)$ there holds
\begin{align*}
\lim_{\rho \to 0} \fint_{B_\rho(z) \cap H_\nu^+(z)} |u(x)-u^+(z)|\mathrm{d}x=0, \quad \lim_{\rho \to 0} \fint_{B_\rho(z) \cap H_\nu^-(z)} |u(x)-u^-(z)|\mathrm{d}x=0.
\end{align*}
The triplet $(u^+(z),u^-(z),\nu_u(z))$ is uniquely determined up to a change of sign of $\nu_u(z)$ and an interchange of $u^+(z)$ and $u^-(z)$. The vector $\nu$ is normal to $S(u)$ in the sense that, if $S(u)$ is represented by (\ref{rectifiability}), then $\nu(z)$ is the normal to $\Gamma_i$ for $\mathcal{H}^{d-1}$ a.e. $z \in \Gamma_i$. In particular it follows that $\nu_u(z) = \pm \nu_v(z)$ for $\mathcal{H}^{d-1}$ a.e. $z \in S(u) \cap S(v)$ and $u,v \in BV(\Omega)$. 
We denote by $\nabla u$ the approximate differential of $u$ at $z \in \Omega$ in the sense that
\begin{align*}
\lim_{\rho \to 0} \fint_{B_\rho(z)}  \frac{\left| u(x)-u(z) -  \nabla u(z) (x-z)\right|}{\left|x-z\right|} \mathrm{d}x = 0.
\end{align*}
For any function $u \in BV(\Omega)$ there holds
\begin{align*}
Du = \nabla u \mathcal{L}^d +  (u^+-u^-)\otimes\nu_u \mathcal{H}^{d-1}_{\lfloor_{S(u)}} + D^c u.
\end{align*}
We say that $u \in BV(\Omega)$ is a special function of bounded variation if the singular part is given by $(u^+-u^-)\otimes\nu_u \mathcal{H}^{d-1}_{\lfloor_{S(u)}}$, i.e.
\begin{align*}
Du = \nabla u\mathcal{L}^d + (u^+-u^-)\otimes\nu_u \mathcal{H}^{d-1}_{\lfloor_{S(u)}}.
\end{align*}
In other words $D^cu=0$.
We denote by $SBV^2(\Omega)$ the space of functions $u \in SBV(\Omega)$ such that
\begin{align*}
\nabla u \in L^2(\Omega;\mathbb{R}^d) \text{ and } \mathcal{H}^{d-1}(S(u)\cap \Omega) < +\infty.
\end{align*}
We also define the space of $GSBV^2(\Omega)$ of generalized $SBV^2(\Omega)$ as the set of all measurable functions $u : \Omega \to [-\infty,+\infty]$ such that for any $T >0 $ $u_T = (u\vee T)\wedge (-T) \in SBV^2(\Omega) $.
If $u \in GSBV^2(\Omega) \cap L^1(\Omega)$ then $u$ has approximate gradient a.e. in $\Omega$, moreover, as $T \to \infty$,
\begin{align*}
&\nabla u_T(x) \to \nabla u(x) \text{ for } \mathcal{L}^d \text{ a.e. $x$ in } \Omega, \text{   and   } |\nabla u_T(x)| \uparrow |\nabla u(x)| \text{ for } \mathcal{L}^d  \text{ a.e. $x$ in } \Omega ,\\
&S(u_T) \subset S(u), \mathcal{H}^{d-1}(S(u_T)) \to  \mathcal{H}^{d-1}(S(u)) \text{ and } \nu_{u_T} = \nu_u \text{ for }  \mathcal{H}^{d-1}\text{ a.e. $x$ in } S(u_T).
\end{align*}

\medskip

We state Besicovitch's Covering Theorem (see e.g. \cite{evans2015measure,federer2014geometric}), since it will be used in the construction for the upper bound.

\begin{theorem}[Besicovitch's Covering Theorem] Let $\mu $ be a positive radon measure on $\Omega$, and let $\mathcal{Q}$ be a collection of closed cubes which covers finely $\Omega$. Then there exists a disjoint and (finite or) countable family $\{Q_i\}_i \subset \mathcal{Q}$ such that 
\begin{align*}
\left(\Omega \setminus \bigcup_i Q_i\right)=0.
\end{align*}
\end{theorem}

Next we recall the definition and some basic properties of $\Gamma$-convergence. We refer to Braides \cite{braides2002gamma} or Dal Maso \cite{dal2012introduction} for a more detailed discussion of this topic.

\medskip

Let $X$ be a metric space equipped with a distance $d$. In what follows $\{F_n\}_n$ will be a sequence of functionals on $X$, i.e. $F_n : X \to \overline{\mathbb{R}}$ and $F : X \to \overline{\mathbb{R}}$.

\begin{definition}[\textbf{$\Gamma$-convergence}]\label{Gammaconvergence definition} We say that the sequence $\{F_n\}_n $ $\Gamma$-converges to $F$ if for all $x \in X$ we have
\begin{itemize}
\item[(i)] For every sequence $\{x_n\}_n \subset X $ converging to $x$ we have that 
\begin{align*}
F(x) \leq \liminf_{n \to \infty} F_n(x_n);
\end{align*}
\item[(ii)] There exists a sequence $\{x_n\}_n \subset X$ converging to $x$ such that
\begin{align*}
F(x) \geq \limsup_{n \to \infty} F_n(x_n).
\end{align*}
\end{itemize}
The function $F$ is called the $\Gamma$-limit of $\{F_n\}_n$ and we write
\begin{align*}
\Gamma\text{-}\lim_{n \to \infty} F_n(x) = F(x).
\end{align*}
\end{definition}

\begin{remark}\label{F'Remark} If we define the functionals $\displaystyle F^{\prime}=\Gamma\text{-}\liminf_{n \to \infty} F_n : X \to \overline{\mathbb{R}}$ and  $ \displaystyle F^{\prime\prime}=\Gamma\text{-}\limsup_{n \to \infty} F_n : X \to \overline{\mathbb{R}}$ by 
\begin{align*}
&F^{\prime}(x)=\Gamma\text{-}\liminf_{n \to \infty} F_n(x) = \inf\Big\{\liminf_{n \to \infty} F_n(x_n) : x_n \to x\Big\},\\ &F^{\prime\prime}(x)=\Gamma\text{-}\limsup_{n \to \infty}F_n(x)= \inf\Big\{\limsup_{n \to \infty} F_n(x_n) : x_n \to x\Big\},
\end{align*}
we have that Definition \ref{Gammaconvergence definition} is equivalent to $\displaystyle\Gamma\text{-}\liminf_{n \to \infty} F_n(x)=\Gamma\text{-}\limsup_{n \to \infty} F_n(x)$ for all $x\in X$. This characterization will be important, since $\displaystyle\Gamma\text{-}\liminf_{n \to \infty} F_n(x)$ and $\displaystyle\Gamma\text{-}\limsup_{n \to \infty} F_n(x)$ defined above always exist and they can be studied separately. $\displaystyle\Gamma\text{-}\liminf_{n \to \infty} F_n(x)$ can be thought of as a lower limit and $\displaystyle\Gamma\text{-}\limsup_{n \to \infty} F_n(x)$ can be thought of as an upper limit to $F$. 
\end{remark}

Next we describe the embedding of the discrete functions into a common function space by interpolation. For general treatment of discrete-to-continuum convergence see \cite{ABC,braides2014discrete}.

\begin{definition}[Discrete functions and discrete-to-continuum convergence]\label{DiscreteDefintion} A function $u : \varepsilon\mathbb{Z}^d \to T$ is identified with its piecewise constant interpolation on the lattice $\varepsilon\mathbb{Z}^d$ given by
 \begin{align*}
 u(x) = \begin{cases} u(i) &x \in Q_\varepsilon(i),  i \in \varepsilon\mathbb{Z}^d \cap \Omega \\
 0 &\text{otherwise.}
 \end{cases}
 \end{align*}
 Note that in this way every such function can be seen as an element of $L^1(\Omega;T)$ (or $L^1_{\mathrm{loc}}(\mathbb{R}^d;T)$). We denote the space of piecewise constant functions associated to the lattice $\varepsilon\mathbb{Z}^d \cap \Omega$ taking values in $T$ by
 \begin{align*}
 \mathcal{PC}_{\varepsilon}(\Omega;T):= \left\{u : \Omega \to T : u \text{ is constant on } Q_\varepsilon(i), i \in \varepsilon\mathbb{Z}^d \cap \Omega  \right\}.
\end{align*}  
We say that a sequence of functions $\{u_\varepsilon\}_\varepsilon$, $u_\varepsilon : \varepsilon\mathbb{Z}^d \cap \Omega \to T$ converges to a function $u \in L^1(\Omega;T)$ strongly in $L^1(\Omega)$ if the sequence of piecewise constant interpolations (still denoted by $\{u_\varepsilon\}_\varepsilon$) converges to $u$ strongly in $L^1(\Omega)$. For $i \in \varepsilon\mathbb{Z}^d$ and $u : \varepsilon\mathbb{Z}^d \to T$ we set $u_i = u(i)$.
\end{definition}
\section{The Main Theorem}

In this section we state the setting of the problem, recall some known results and state the main theorem.

\bigskip

 Let $V\subset \mathbb{Z}^d$ containing the standard orthonormal basis $\{e_1,\ldots,e_d\}$ and let $c^\varepsilon_{i,\xi}\geq 0$ satisfy
\begin{itemize}
\item[(H1)] $\displaystyle\inf_{i,k} c^\varepsilon_{i,e_k} \geq c >0$ for all $\varepsilon >0$.
\item[(H2)] $c_{i,\xi}^\varepsilon =0 \quad \forall \xi \in \mathbb{Z}^d \setminus V$. 
\item[(H3)] $\displaystyle\sup_{i} c_{i,\xi}^\varepsilon \leq c^* <+\infty$ for all $\varepsilon>0$.
\end{itemize}
For such coefficients $c_{i,\xi}^\varepsilon$ recall the definition of spin energy $E_\varepsilon : L^1(\Omega) \times \mathcal{A}(\Omega) \to [0,+\infty]$ given by
\begin{align}\label{SpinEnergy}
E_\varepsilon(u,A) = \begin{cases}\displaystyle\frac{1}{4}\sum_{\xi \in V} \underset{i+\varepsilon\xi \in Z_\varepsilon(\Omega)}{\sum_{i \in Z_\varepsilon(A)}}\varepsilon^{d-1} c^\varepsilon_{i,\xi}(u_{i+\varepsilon\xi}-u_i)^2 & u \in \mathcal{PC}_\varepsilon(\Omega;\{-1,+1\})\\+\infty &\text{otherwise,} 
\end{cases}
\end{align}
and we define the 'weak membrane energies' $F_\varepsilon : L^1(\Omega) \times \mathcal{A}(\Omega) \to [0,+\infty]$ by

\begin{align} \label{DiscreteMSENergies}
F_\varepsilon(u,A) = \begin{cases}\displaystyle\sum_{\xi \in V}\underset{i+\varepsilon\xi \in Z_\varepsilon(\Omega)}{\sum_{i \in Z_\varepsilon( A)} }\varepsilon^{d} W^{i,\xi}_\varepsilon\left(D^\xi_\varepsilon u(i)\right) & u \in \mathcal{PC}_\varepsilon(\Omega)\\+\infty &\text{otherwise,} 
\end{cases}
\end{align}
where $W^{i,\xi}_\varepsilon : \mathbb{R} \to [0,+\infty)$ is defined by 
\begin{align*}
W^{i,\xi}_\varepsilon(z) = \frac{c^\varepsilon_{i,\xi}}{\varepsilon} \wedge z^2.
\end{align*}
We write $F_\varepsilon(u) = F_\varepsilon(u,\Omega)$. Moreover for $A \in \mathcal{A}^{reg}(\Omega)$,  $\eta,\varepsilon >0$ and $u \in \mathcal{PC}_\varepsilon(\Omega)$ we define
\begin{align*}
m_{\varepsilon,\eta}^F(u,A) = \inf\left\{F(v,A) : v \in \mathcal{PC}_\varepsilon(\Omega),  v_i = u_i \text{ for all } i \in Z_\varepsilon (A_\eta \cup A^c)\right\}.
\end{align*}

 By \cite{braides2016optimal} we have that up to subsequences $E_\varepsilon $ $\Gamma$-converges with respect to the strong $L^1(\Omega)$ topology to an energy $E :L^1(\Omega) \times \mathcal{A}(\Omega) \to [0,+\infty]$ defined by
\begin{align}\label{limit energy}
E(u,A) = \begin{cases}\displaystyle\int_{S(u)\cap A} \varphi(x,\nu_u(x))\mathrm{d}\mathcal{H}^{d-1} & u \in BV(\Omega;\{-1,+1\})\\
+\infty &\text{otherwise,}
\end{cases}
\end{align}
where $\varphi : \Omega \times S^{d-1} \to [0,+\infty)$ is given by
\begin{align}\label{Spindensity}
\varphi(x_0,\nu) = \lim_{\rho \to 0} \frac{1}{\rho^{d-1}}\lim_{\eta \to 0}\limsup_{\varepsilon \to 0} \inf \Bigg\{&E_\varepsilon(v, Q^\rho_\nu(x_0)) : v \in \mathcal{PC}_\varepsilon(\mathbb{R}^d;\{-1,+1\}) \\& \nonumber: v_i = (u_{x_0,\nu})_i \text{ for all } i \in  Z_\varepsilon\left(\left(Q_\nu^\rho(x_0)\right)_\eta \right)\Bigg\}.
\end{align}

The goal of this article is to prove the following theorem

\begin{theorem} \label{MainTheorem} Let $c_{i,\xi}^\varepsilon \geq 0$ satisfy (H1)-(H3) and let $F_\varepsilon : L^1(\Omega) \to [0,+\infty]$ be given by (\ref{DiscreteMSENergies}). Then there exists a subsequence $\{\varepsilon_k\}_k \subset \{\varepsilon\}$ such that $F_{\varepsilon_k}$ $\Gamma$-converges with respect to the strong $L^1(\Omega)$-topology to the functional $F : L^1(\Omega) \to [0,+\infty]$ defined by
\begin{align}\label{AnistropicMSEnergy}
F(u)=  \begin{cases} \displaystyle\int_\Omega f(x,\nabla u) \mathrm{d}x + \int_{S(u)}\varphi(x,\nu_u(x))\mathrm{d}\mathcal{H}^{d-1} &u \in GSBV^2(\Omega)\cap L^1(\Omega)\\+\infty &\text{otherwise,}
\end{cases}
\end{align}
where $f : \Omega \times\mathbb{R}^d  \to [0,+\infty)$ is a quasiconvex Carath\'eodory function satisfying 
\begin{align} \label{fbounds}
c(|\zeta|^2 -1) \leq f(x_0,\zeta) \leq C (|\zeta|^2+1)
\end{align}
for some $0<c<C$ and is given by
\begin{align}\label{Definition f}
f(x_0,\zeta)= \lim_{\rho \to 0} \frac{1}{\rho^d} \lim_{\eta \to 0} \limsup_{k \to\infty} m_{\varepsilon_k,\eta}^{F_\varepsilon}(\zeta \cdot, Q_\rho^\nu(x_0))
\end{align}
 and  $\varphi : \Omega \times  S^{d-1} \to [0,+\infty)$ is given by (\ref{Spindensity}) with $\{\varepsilon_k\}$ in place of $\{\varepsilon\}$.
\end{theorem}

The proof of this theorem follows once we have proved Propositions \ref{GammaliminfMS} and \ref{GammalimsupLemma}, which will be established in the following section.

\begin{figure}[htp]
\centering
\includegraphics{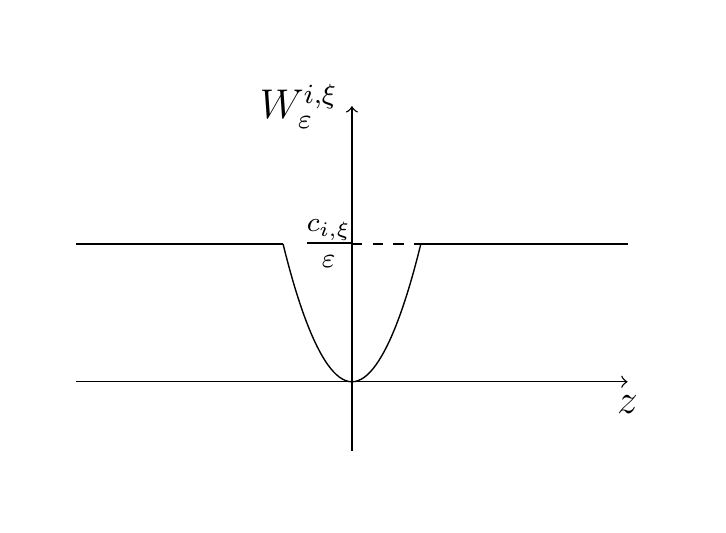}
\caption{The interaction potential for $\varepsilon >0$ between the points $ i$ and $i+\varepsilon\xi$ }
\end{figure}

\begin{remark} \label{RemarkDensityAC} If we assume that $\inf c_{i,\xi}^\varepsilon \geq c$, where the $\inf$ is taken over $\xi \in V, \varepsilon>0$, then the density of the absolutely continuous part can be computed explicitly.
In fact, since we have that $c\leq c^\varepsilon_{i,\xi} \leq c^*$ for all $\xi \in V$, with $0<c<c^*$ we have that
\begin{align}\label{densityAC}
f(x,\zeta) = \sum_{\xi \in V} | \xi\cdot\zeta|^2.
\end{align}
This follows from comparison with the energies
\begin{align*}
F^*_\varepsilon(u) = \sum_{\xi \in V} \underset{i+\varepsilon\xi \in Z_\varepsilon(\Omega)}{\sum_{i \in Z_\varepsilon( \Omega)}}\varepsilon^d W^{\varepsilon,*}_\xi(D^\xi_\varepsilon u(i)), \quad F_{\varepsilon,*}(u) = \sum_{\xi \in V} \underset{i+\varepsilon\xi \in Z_\varepsilon(\Omega)}{\sum_{i \in Z_\varepsilon(\Omega)}}\varepsilon^d W^{\varepsilon}_{\xi,*}(D^\xi_\varepsilon u(i))
\end{align*}

 whose interaction potentials are 
\begin{align*}
W^{\varepsilon,*}_\xi(z) = \frac{c^*}{\varepsilon}\wedge z^2, \quad W^{\varepsilon}_{\xi,*}(z) = \frac{c}{\varepsilon}\wedge z^2,
\end{align*}
whose $\Gamma$-limits are known by \cite{chambolle1999finite} Theorem 1, and whose density of the absolutely continuous part is given by (\ref{densityAC}). We have 
\begin{align*}
F^*_\varepsilon(u) \leq F_\varepsilon(u) \leq F_{\varepsilon,*}(u),
\end{align*}
thus the same relation hold for the $\Gamma$-limits. Fixing a point $x_0 \in \Omega$, that is a Lebesgue point for the measures $F^*(u,\cdot),F_*(u,\cdot),F(u,\cdot)$ and such that the Radon-Nikodym derivative converges to the density of the absolutely continuous part (note that this property is satisfied for all points but a Lebesgue-Null-set) (\ref{densityAC}) follows. This completely characterizes $f$. Section 5 is devoted to the characterization of $f$ in the case, where the number of interactions may vary from point to point. In that case a more careful analysis is needed.
\end{remark}
We set $H_\varepsilon : \mathcal{PC}_\varepsilon(\Omega) \times \mathcal{A}(\Omega) \to [0,+\infty]$   
\begin{align}\label{Heps}
H_\varepsilon(u,A) = \sum_{i \in Z_\varepsilon(A)}  \sum_{\xi \in V}\varepsilon^d \mathrm{1}_{c_{i,\xi}^\varepsilon >0} |D^\xi_\varepsilon u(i)|^2.
\end{align}
Note that $ F_\varepsilon(u,A) \leq H_\varepsilon(u,A)$ for all $(u,A) \in \mathcal{PC}_\varepsilon(\Omega) \times \mathcal{A}(\Omega)$. By \cite{alicandro2004general} we have that there exists a subsequence $\{\varepsilon_k\}_k$ and a caratheodory function $h$ such that
\begin{align*}
\Gamma\text{-}\lim_{k \to \infty} H_{\varepsilon_k}(u,A) =  \int_A h(x,\nabla u)\mathrm{d}x = H(u,A),
\end{align*}
where by \cite{alicandro2004general}, Theorem 3.1 and Corollary 3.11
\begin{align}\label{Definition h}
h(x_0,\zeta) = \lim_{\rho \to 0} \frac{1}{\rho^d} \limsup_{k \to \infty}  m^{H_{\varepsilon_k}}_{\varepsilon_k,\varepsilon_k}\left(\zeta\cdot,Q^\nu_\rho(x_0)\right) 
\end{align}
for all $\zeta \in \mathbb{R}^d$ and a.e. $x \in \Omega$.

\begin{remark} Note that $\varphi$ and $f$ may depend on the subsequence $\{\varepsilon_k\}_k \subset \{\varepsilon\}$ that has been chosen. By the compactness properties of $\Gamma$-convergence (see \cite{dal2012introduction}) the $\Gamma$-limit is known to always exist under subtraction of a subsequence. Since the coefficients depend on $\varepsilon$ one can construct examples like  $V=\{e_1,e_2,e_1+e_2,e_1-e_2\}$ and
\begin{align*}
c_{i,\xi}^{\varepsilon_n} = \begin{cases} 1 &\text{if } \xi\in\{e_1,e_2\} \text{ or } n \text{ odd} \\
0 &\text{otherwise},
\end{cases}
\end{align*}
where $\varepsilon = \varepsilon_n \to 0$ as $n\to \infty$.
The surface energy densities for the even and for the odd subsequence are pictured in Fig.~\ref{FigEnergyDensities} on the left and on the right respectively.
\begin{figure}[htp]
\centering
\includegraphics{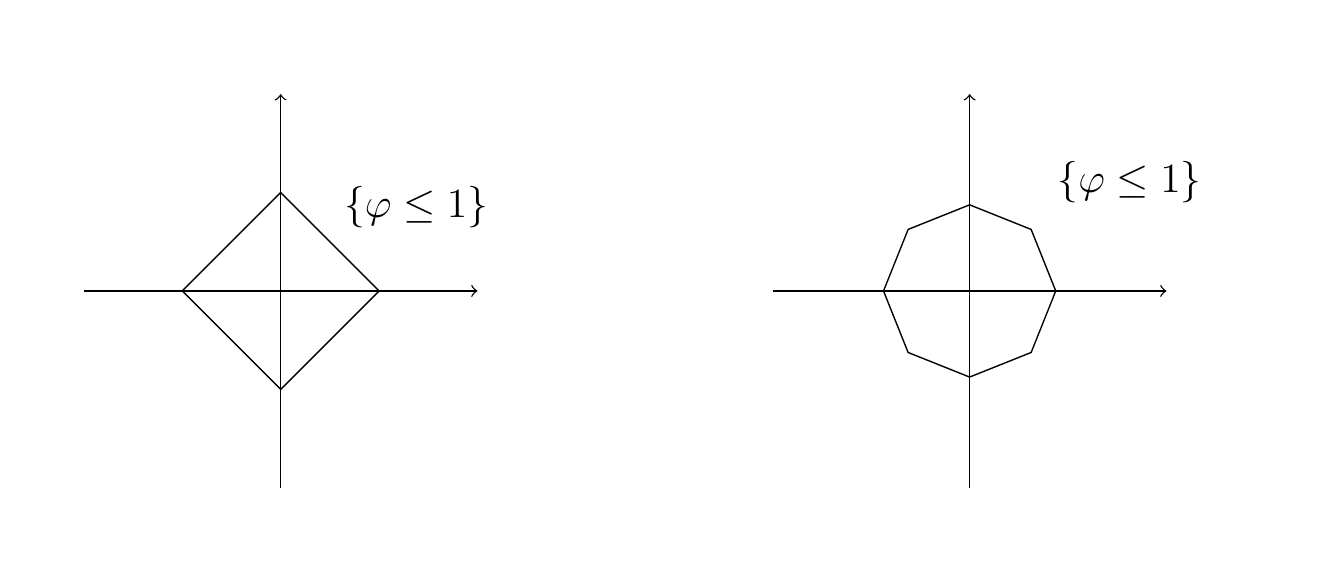}
\caption{The different surface energy densities for different subsequences}\label{FigEnergyDensities}
\end{figure}
\end{remark}
\begin{definition}
Let $c_{i,\xi}, i \in \mathbb{Z}^d, \xi \in V$, $V \subset \mathbb{Z}^d$ finite and containing the standard orthonormal basis be given. We say that the $c_{i,\xi}$ are periodic for some period $T\in \mathbb{N}$ if 
\begin{align*}
c_{i+Tz,\xi}=c_{i,\xi} \text{ for all } i,z \in \mathbb{Z}^d \text{ and all } \xi \in V.
\end{align*}
For such coefficients we set $c_{i,\xi}^\varepsilon = c_{\frac{i}{\varepsilon},\xi}$
\end{definition}
\begin{corollary} Let $T \in \mathbb{N}$ and let $c_{i,\xi}\geq 0$ satisfy (H1)-(H3) and be $T$-periodic. Then the family of functionals $F_\varepsilon:L^1(\Omega) \to [0,+\infty]$ $\Gamma$-converges with respect to the strong $L^1(\Omega)$-topology to the functional $F: L^1(\Omega)\to [0,+\infty]$ given by (\ref{AnistropicMSEnergy}), where $f(x_0,\zeta)=f(\zeta)$ does not depend on $x_0$ and is given by
\begin{align} \label{fhom}
f(\zeta) = \lim_{T \to \infty} \frac{1}{T^{d}}\inf\left\{H_1(v,Q_T) : v \in \mathcal{PC}_1(\mathbb{R}^d), v_i = \zeta i \text{ for all } i \notin Z_1(Q_T) \right\}
\end{align}
 and 
$
\varphi(x_0,\nu) = \varphi(\nu)
$ does not depend on $x_0$ and is given by 
\begin{align}\label{Varphihom}
\varphi(\nu) = \lim_{T\to \infty} \frac{1}{T^{d-1}}\inf \Bigg\{&E_1(v, Q_T^\nu) : v \in \mathcal{PC}_1(\mathbb{R}^d;\{-1,+1\}),\\&\nonumber \qquad v_i = (u_{0,\nu})_i \text{ for all } i \notin  Z_1\left(Q_T^\nu\right)\Bigg\}.
\end{align}
\end{corollary}
\begin{proof} By Theorem \ref{MainTheorem} we have that there exists a subsequence $\{\varepsilon_k\}_k$ and $F:L^1(\Omega)\to [0,+\infty]$ of the form (\ref{AnistropicMSEnergy}) the such that $F_{\varepsilon_k}$ $\Gamma$-converges with respect to the strong $L^1(\Omega)$ topology to the functional $F$. By Proposition \ref{Prop Char BulkDensity} we have that $f=h$ with $h$ given by (\ref{Definition h}). By \cite{alicandro2004general}, Theorem 4.1 we have that $h$ and therefore $f$ is given by (\ref{fhom}). $\varphi$ coincides with the density of the $\Gamma$-limit of the spin energies given by (\ref{limit energy}). By \cite{alicandro2016local},Theorem 4.7 we have that $\varphi$ is independent of the first variable and given by (\ref{Varphihom}). Since the $F$ is independent of the chosen subsequence we have that actually the whole sequence $F_\varepsilon$ $\Gamma$-converges with respect to the strong $L^1(\Omega)$-topology to $F$.
\end{proof}

\begin{remark} Note that if we have periodic interaction coefficients of finite range by \cite{kreutz2018thesis} Theorem 3.0.5 we have that the surface energy density is crystalline, that is the set $\{\varphi\leq 1\}$ is a convex polyhedron. This implies that for fixed period of the coefficients $c_{i,\xi}$ the (isotropic) Mumford Shah functional can only be approximated up to a certain error (depending on the period), since certain directions are preferred due to the crystallinity of the surface energy density. In \cite{braides2016optimal} it is proved that there exist periodic microstructures whose periods tend to $\infty$ and whose homogenized surface energy densities approximate arbitrarily well the energy density $\varphi(\nu)=|\nu|$.
\end{remark}
\section{Asymptotic Analysis}

This section contains the proof of the main theorem. The equi-coercivity follows by using (H1) and estimating from below with funtionals that are coercive with respect to the strong $L^1$-topology. The lower bound follows by a blowup-argument while using a discrete coarea formula to reduce the class of admissible competitors for the cell formula to piecewise constant functions taking only two values. The upper bound is done in two steps. First we prove a density result and for that class we construct an (explicit) recovery sequence using a Besicovitch covering argument.
\begin{lemma}\label{Compactnesslemma} Let $\{u_\varepsilon\}_\varepsilon \subset \mathcal{PC}_\varepsilon(\Omega)$ be such that
\begin{align*}
\sup_{\varepsilon >0} F_\varepsilon(u_\varepsilon) < +\infty,\quad \sup_{\varepsilon >0} ||u_\varepsilon||_\infty < +\infty.
\end{align*}
Then there exists a subsequence $\{u_{\varepsilon_k}\}_k \subset \{u_\varepsilon\}_\varepsilon$ and a function $u \in SBV^2(\Omega) \cap L^\infty(\Omega)$ such that $u_{\varepsilon_k}$ converges to $u$ with respect to the strong $L^1(\Omega)$ topology.
\end{lemma}
\begin{proof} The proof follows by applying  \cite{chambolle1999finite}, Lemma 1 and noting that
\begin{align*}
F_\varepsilon(u_\varepsilon) \geq \sum_{k=1}^d \sum_{i \in Z_\varepsilon(\Omega)} \varepsilon^d W_\varepsilon(D^{e_k}_\varepsilon u(i)),
\end{align*}
where $W_\varepsilon : \mathbb{R} \to \mathbb{R}$ is defined by 
\begin{align*}
W_\varepsilon(z) = \frac{c_* }{\varepsilon}\wedge z^2,
\end{align*}
with $c_* = \displaystyle \inf_{i \in \varepsilon\mathbb{Z}^d,k\in \{1,\ldots,d\} }c^\varepsilon_{i,e_k} >0$.

\end{proof}

\begin{proposition}\label{GammaliminfMS}
\begin{align*}
F^{\prime}(u) \geq F(u).
\end{align*}
\end{proposition}
\begin{proof} It suffices to consider $u_\varepsilon \to u$ in $L^1(\Omega)$ such that 
\begin{align*}
\liminf_{\varepsilon \to 0} F_\varepsilon(u) <+\infty.
\end{align*}
Up to subsequences we may suppose that $\displaystyle \liminf_{n \to \infty} F_{\varepsilon_n}(u_{\varepsilon_n})= \lim_{\varepsilon \to 0} F_{\varepsilon_n}(u_{\varepsilon_n})$. Since 
$
F_\varepsilon((u_\varepsilon)_T) \leq F_\varepsilon(u)
$ and $\displaystyle\lim_{T \to \infty} F(u_T) =F(u)$
we assume furthermore, that $\displaystyle\sup_{\varepsilon >0} || u_\varepsilon ||_\infty \leq C <+\infty $. By Lemma \ref{Compactnesslemma} we have that $u \in SBV^2(\Omega)$. Consider now the family of measures
\begin{align*}
\mu_n = \sum_{\xi \in V}\sum_{i \in Z_\varepsilon(\Omega) } \varepsilon_n^d W_{i,\xi}^{\varepsilon_n}\left(D^\xi_{\varepsilon_n} u(i)\right) \delta_{i}.
\end{align*}
Note that $\displaystyle \sup_{n} \mu_n(\Omega) = F_{\varepsilon_n}(u_{\varepsilon_n}) <+\infty$ and therefore up to passing to a further subsequence (not relabbeled), we may suppose that there exists $\mu \in \mathcal{M}_b(\Omega)$ such that $\mu_n \overset{*}{\rightharpoonup} \mu$. By the Radon-Nikodym Theorem we may decompose $\mu$ into three mutually disjoint non-negative measures such that
\begin{align*}
\mu = g\mathcal{L}^d + q\mathcal{H}^{d-1}\lfloor_{S(u)} + \mu_s.
\end{align*}
We complete the proof if we show that

\begin{itemize}
\item[i)] $q(x_0) \geq \varphi(x_0,\nu_u(x_0)) \text{ for } \mathcal{H}^{d-1} \text{-a.e. } x_0 \in S(u). $
\item[ii)] $g(x_0) \geq f(x_0,\nabla u(x_0)) \text{ for } \mathcal{L}^d \text{-a.e. } x_0 \in \Omega$.
\end{itemize}
The claim follows using Lemma \ref{LiminfSurfaceLemma} and Lemma \ref{LiminfBulkLemma}.
\end{proof}
\begin{lemma} \label{LiminfSurfaceLemma}
\begin{align*}
q(x_0) \geq \varphi(x_0,\nu_u(x_0)) \text{ for } \mathcal{H}^{d-1} \text{-a.e. } x _0\in S(u). 
\end{align*}
\end{lemma}
\begin{proof}
Fix $x_0 \in S(u)$ and assume $q(x_0) < +\infty$, since otherwise there is nothing to prove. To simplify notation we write $\nu=\nu_u(x_0)$, $\varepsilon=\varepsilon_n$,$u^+(x_0)=z_1,u^-(x_0)=z_2$ and $u_{0}= u_{x_0,\nu}^{u^+(x_0),u^-(x_0)}$.
By the properties of $SBV$ functions and radon measures we have that
\begin{itemize}
\item[a)] $\displaystyle \lim_{\rho \to 0}\fint_{Q^\nu_\rho(x_0)} |u-u_{0}|\mathrm{d}x=0$, 
\item[b)] $\displaystyle q(x_0)= \lim_{\rho \to 0} \frac{\mu(Q_\rho^\nu(x_0))}{\rho^{d-1}} $.
\end{itemize}
for $\mathcal{H}^{d-1}$-a.e. $x_0 \in S(u)$. Thus it suffices to prove Lemma \ref{GammaliminfMS} for points satisfying a) and b).
Fix such a $x_0 \in S(u)$ and $\rho \to 0$ such that $\mu(\partial Q^\nu_\rho(x_0))=0$. By the weak convergence of $\mu_n$ to $\mu$ we have that
\begin{align*}
q(x_0)&= \lim_{\rho \to 0} \frac{\mu(Q_\rho^\nu(x_0))}{\rho^{d-1}} = \lim_{\rho \to 0} \frac{1}{\rho^{d-1}} \lim_{n \to \infty} \mu_n(Q^\nu_\rho(x_0)) \\&= \lim_{\rho \to 0 } \frac{1}{\rho^{d-1}} \lim_{\varepsilon \to 0}\sum_{\xi \in V} \sum_{i \in Z_\varepsilon( Q^\nu_\rho(x_0))} \varepsilon^d W^\varepsilon_{i,\xi}\left(D^\xi_\varepsilon u(i)\right).
\end{align*}
Since the limit exist and is finite we have that for $\rho $ and $\varepsilon $ small enough there holds
\begin{align*}
\sum_{\xi \in V} \sum_{i \in Z_\varepsilon( Q^\nu_\rho(x_0))} \varepsilon^d W^\varepsilon_{i,\xi}\left(D^\xi_\varepsilon u(i)\right) \leq C\rho^{d-1}.
\end{align*}
Fix $\eta >0$ We construct $v_\varepsilon \in \mathcal{PC}_\varepsilon(\mathbb{R}^d;\{z_1,z_2\})$ such that $v_\varepsilon = u_{0}$ on $(Q^\nu_\rho(x_0))_\eta $ and
\begin{align}\label{wboundv} 
F_\varepsilon(v_\varepsilon,Q^\nu_\rho(x_0)) \leq F_\varepsilon(u_\varepsilon,Q^\nu_\rho(x_0)) + o(\rho^{d-1})
\end{align}
Assume without loss of generality, that $z_1 <z_2$ and define for $t \in (z_1,z_2)$ 
\begin{align*}
I_{t,\varepsilon} = \Bigg\{(i,\xi) \in Z_\varepsilon( Q_\nu^\rho(x_0)) \times V, (u_{\varepsilon})_{i+\xi} \wedge (u_{\varepsilon})_i \leq t \leq&  (u_{\varepsilon})_{i+\xi} \vee (u_{\varepsilon})_i,\\& |(u_{\varepsilon})_{i+\xi}-(u_{\varepsilon})_i|\leq \sqrt{c^\varepsilon_{i,\xi} \varepsilon}  \Bigg\}.
\end{align*}
Using Fubini's Theorem and H\"older's Inequality we obtain that there exists $t_\varepsilon \in (z_1,z_2) $ such that $\mathrm{dist}(\{z_1,z_2\},t_\varepsilon)\geq c>0$ and
\begin{align*}
\frac{1}{2}(z_2-z_1)\#I_{t_\varepsilon,\varepsilon} &\leq \int_{z_1}^{z_2} \#I_{t,\varepsilon}\mathrm{d}t = \int_{z_1}^{z_2} \sum_{\xi \in V}\sum_{i \in Z_\varepsilon( Q^\rho_\nu(x))} \mathrm{1}_{I_{t,\varepsilon}}(i,\xi) \mathrm{d}t \\&= \sum_{\xi\in V }\underset{|(u_{\varepsilon})_{i+\xi}-(u_{\varepsilon})_i|\leq \sqrt{c_{i,\xi}^\varepsilon \varepsilon}}{\sum_{i \in Z_\varepsilon( Q^\rho_\nu(x))} } |(u_{\varepsilon})_{i+\xi}-(u_{\varepsilon})_i| \\& \leq \varepsilon^{1-d} \Bigg(\sum_{\xi\in V}\sum_{i \in Z_\varepsilon( Q^\rho_\nu(x))} \varepsilon^d \Bigg)^{\frac{1}{2}} \cdot \Bigg(\sum_{\xi\in V}\underset{|(u_{\varepsilon})_{i+\xi}-(u_{\varepsilon})_i|\leq \sqrt{c^\varepsilon_{i,\xi} \varepsilon}}{\sum_{i \in Z_\varepsilon( Q^\rho_\nu(x))} } \varepsilon^d \left| D^\xi_\varepsilon u(i)\right|^2 \Bigg)^{\frac{1}{2}}\\&\leq C\varepsilon^{1-d} \rho^{\frac{d}{2}} F_\varepsilon(u_\varepsilon,Q_\nu^\rho(x))^{\frac{1}{2}} \left(\#V\right)^{\frac{1}{2}} \leq C\varepsilon^{1-d} \rho^{d-\frac{1}{2}}\left(\#V\right)^{\frac{1}{2}}
\end{align*}
Now defining $w_\varepsilon \in \mathcal{PC}_\varepsilon(\mathbb{R}^d;\{z_1,z_2\})$ by
\begin{align*}
w_\varepsilon(i) = (z_2-z_1) \mathrm{1}_{\{v_i > t_\varepsilon\}}(i) + z_1
\end{align*}
we have that 
\begin{align*}
 &F_\varepsilon(w_\varepsilon,Q^\nu_\rho(x_0))= \sum_{\xi \in V} \sum_{i \in Z_\varepsilon( Q_\nu^\rho(x))} \varepsilon^d W^{i,\xi}_\varepsilon\left(D^\xi_\varepsilon w(i)\right) \\=& \sum_{\xi \in V} \underset{(i,\xi) \in I_{t_\varepsilon,\varepsilon}}{\sum_{i \in Z_\varepsilon( Q_\nu^\rho(x))}} \varepsilon^d W^{i,\xi}_\varepsilon\left(D^\xi_\varepsilon w(i)\right) + \sum_{\xi \in V} \underset{(i,\xi) \notin I_{t_\varepsilon,\varepsilon}}{\sum_{i \in Z_\varepsilon( Q_\nu^\rho(x))}} \varepsilon^d W^{i,\xi}_\varepsilon\left(D^\xi_\varepsilon w(i)\right) \\ \leq& C \rho^{d-\frac{1}{2}}\left(\#V\right)^{\frac{1}{2}} +\sum_{\xi \in V} \underset{(i,\xi) \notin I_{t_\varepsilon,\varepsilon}}{\sum_{i \in Z_\varepsilon ( Q_\nu^\rho(x))}} \varepsilon^d W^{i,\xi}_\varepsilon\left(D^\xi_\varepsilon w(i)\right). 
\end{align*}
Now if $(i,\xi) \notin I_{t_\varepsilon,\varepsilon}$ we have either $ t_\varepsilon \notin ((u_{\varepsilon})_{i+\xi} \wedge (u_{\varepsilon})_i, (u_{\varepsilon})_{i+\xi} \vee (u_{\varepsilon})_i)$ or $|(u_{\varepsilon})_{i+\xi} - (u_{\varepsilon})_i| > \sqrt{c_{i,\xi}\varepsilon}$. In either case we have that
\begin{align*}
 W^{i,\xi}_\varepsilon\left(D^\xi_\varepsilon w(i)\right) \leq  W^{i,\xi}_\varepsilon\left(D^\xi_\varepsilon u(i)\right).
\end{align*}
By $u_\varepsilon \to u$ in $L^1(Q^\nu_\rho(x_0))$  and a) we have that
\begin{align*}
 \lim_{\rho \to 0}\lim_{\varepsilon \to 0}\fint_{Q^\nu_\rho(x_0)} |u_\varepsilon-u_{0}|\mathrm{d}x =0.
\end{align*}
Now
\begin{align*}
\fint_{Q^\nu_\rho(x_0)} |w_\varepsilon-u_{x_0,\nu}|\mathrm{d}x &= \frac{|z_2-z_1|}{\rho^d} \left( \left|\left\{ u_\varepsilon > t_\varepsilon\right\} \cap \left\{ u_{0}=z_1\right\}\right| +  \left|\left\{ u_\varepsilon \leq t_\varepsilon\right\} \cap \left\{ u_{0}=z_2\right\}\right| \right) \\& \leq \frac{|z_2-z_1|}{\rho^d} \left| \left\{|u_\varepsilon-u_{0}|>c\right\} \right| \leq \frac{|z_2-z_1|}{c\rho^d}\fint_{Q^\nu_\rho(x_0)} |u_\varepsilon-u_{0}|\mathrm{d}x
\end{align*}
and therefore 
\begin{align}\label{wconvergence}
 \lim_{\rho \to 0}\lim_{\varepsilon \to 0}\fint_{Q^\nu_\rho(x_0)} |w_\varepsilon-u_{0}|\mathrm{d}x =0.
\end{align}
We now construct $v_\varepsilon \in \mathcal{PC}_\varepsilon(\mathbb{R}^d;\{z_1,z_2\})$ such that $v_\varepsilon =u_{0}$ on $(\partial Q^\nu_\rho(x_0))_\eta$ and 
\begin{align}\label{vlessthanw}
F_\varepsilon(v_\varepsilon,Q^\nu_\rho(x_0)) \leq F_\varepsilon(w_\varepsilon,Q^\nu_\rho(x_0)) + o(\rho^{d-1}).
\end{align}
To this end let $R=\sup_{\xi \in V} |\xi|$, $K_\varepsilon^\delta = \lfloor \frac{\delta\rho}{3R\varepsilon}\rfloor \in \mathbb{N}$ and for $k \in \{0,\ldots,3K_\varepsilon^\delta-1\}$ we set $Q_k = Q_{R\varepsilon k +\eta}^\nu(x_0)$ and $S_k = Q_{3R\varepsilon (k+1) +\eta}^\nu(x_0)\setminus \overline{Q}_{3R\varepsilon k+\eta}^\nu(x_0)$. For any $\varepsilon >0$ there exists $k_\varepsilon\in\{0,\ldots,3K_\varepsilon^\delta-1\}$  such that
\begin{align}\label{cardinalityestimate}\begin{split}
C \int_{Q^\nu_\rho(x_0)}|w_\varepsilon -u_{0}|\mathrm{d}x &= \sum_{k=0}^{K^\delta_\varepsilon-1} \int_{S_k}|w_\varepsilon -u_{0}|\mathrm{d}x \geq  K^\delta_\varepsilon \sum_{i \in Z_\varepsilon( S_{k_\varepsilon})} \varepsilon^d|(w_\varepsilon)_i-(u_{0})_i|  \\&\geq \varepsilon^d |z_2-z_1| K^\delta_\varepsilon \#\{i \in \varepsilon\mathbb{Z}^d \cap  S_{k_\varepsilon} : (w_\varepsilon)_i \neq (u_{0})_i \}.
\end{split}
\end{align}
Now defining 
\begin{align*}
v_\varepsilon(i)=\begin{cases} w_\varepsilon(i) &i \in Z_\varepsilon( Q^\nu_\rho(x_0) \setminus \overline{Q}_{ (3k_\varepsilon+1) }) \\
u_{0}(i) &\text{otherwise}
\end{cases}
\end{align*}
we have for $\varepsilon >0$ small enough
\begin{align*}
F_\varepsilon(v_\varepsilon,Q^\nu_\rho(x_0)) = \sum_{\xi \in V} \sum_{i \in Z_\varepsilon( Q^\nu_\rho(x_0))}\varepsilon^d W^\varepsilon_{i,\xi}\left(D^\xi_\varepsilon v(i)\right) = \sum_{i=1}^4 I_\varepsilon^i
\end{align*}
where
\begin{align*}
\sum_{i=1}^4 I_\varepsilon^i&=\sum_{\xi \in V}\underset{i+\varepsilon\xi \in Z_\varepsilon( Q^\nu_\rho(x_0) \setminus \overline{Q}_{3k_\varepsilon+1}) }{\sum_{i \in Z_\varepsilon( Q^\nu_\rho(x_0) \setminus \overline{Q}_{3k_\varepsilon+1})}} \varepsilon^d W^\varepsilon_{i,\xi}\left(D^\xi_\varepsilon v(i)\right)+ \sum_{\xi \in V}\underset{i+\varepsilon\xi \notin Z_\varepsilon( Q^\nu_\rho(x_0) \setminus \overline{Q}_{3k_\varepsilon+1}) }{\sum_{i \in Z_\varepsilon( Q^\nu_\rho(x_0) \setminus \overline{Q}_{3k_\varepsilon+1})}} \varepsilon^d W^\varepsilon_{i,\xi}\left(D^\xi_\varepsilon v(i)\right) \\&\quad+ \sum_{\xi \in V}\underset{i+\varepsilon\xi \in Z_\varepsilon( Q^\nu_\rho(x_0) \setminus \overline{Q}_{3k_\varepsilon+1}) }{\sum_{i \notin Z_\varepsilon( Q^\nu_\rho(x_0) \setminus \overline{Q}_{3k_\varepsilon+1})}} \varepsilon^d W^\varepsilon_{i,\xi}\left(D^\xi_\varepsilon v(i)\right) + \sum_{\xi \in V}\underset{i+\varepsilon\xi \in Z_\varepsilon( \overline{Q}_{3k_\varepsilon+1}) }{\sum_{i \in Z_\varepsilon( \overline{Q}_{3k_\varepsilon+1})}} \varepsilon^d W^\varepsilon_{i,\xi}\left(D^\xi_\varepsilon v(i)\right) 
\end{align*}
For the first term we have
\begin{align} \label{I1estimate}
I_\varepsilon^1 \leq F_\varepsilon(w_\varepsilon,Q^\nu_\rho(x_0)),
\end{align}
since $w_\varepsilon=v_\varepsilon$ in $Q^\nu_\rho(x_0) \setminus \overline{Q}_{3k_\varepsilon+1}$. For the second term and third term we we have, using (\ref{cardinalityestimate}), 
\begin{align}\label{I2I3estimate}
I_\varepsilon^2 +I^3_\varepsilon \leq O(\varepsilon) +  \frac{C \# V}{\delta\rho} \int_{Q^\nu_\rho(x_0)}|w_\varepsilon- u_{x_0,\nu}|\mathrm{d}x
\end{align}
and, since $v_\varepsilon = u_{x_0,\nu}$ on $\overline{Q}_{3k_\varepsilon+1}$ we have
\begin{align} \label{I4estimate}
I_\varepsilon^4 = \sum_{\xi \in V}\underset{i+\varepsilon\xi \in Z_\varepsilon( \overline{Q}_{3k_\varepsilon+1}) }{\sum_{i \in Z_\varepsilon( \overline{Q}_{3k_\varepsilon+1})}} \varepsilon^d W^\varepsilon_{i,\xi}\left(D^\xi_\varepsilon u_{x_0,\nu}(i)\right) \leq C \eta \rho^{d-2}.
\end{align}
Noting that $\eta <<\rho$, using (\ref{wconvergence}) and (\ref{I1estimate})-(\ref{I4estimate}) we obtain (\ref{vlessthanw}). Note that we have  
\begin{align*}
F_\varepsilon(v_\varepsilon,Q^\nu_\rho(x_0)) \geq m_{\varepsilon,\eta}^{F_\varepsilon}(u_{0},Q^\nu_\rho(x_0))
\end{align*}
and
\begin{align}\label{infeq}
m_{\varepsilon,\eta}^{F_\varepsilon}(u_{0},Q^\nu_\rho(x_0)) = m_{\varepsilon,\eta}^{E_\varepsilon}(u_{x_0,\nu},Q^\nu_\rho(x_0)) 
\end{align}
for all $z_1,z_2 \in \mathbb{R}, z_1\neq z_2$ and $\varepsilon>0$ small enough. (Exactly when $2 \wedge |z_1-z_2| > \sqrt{c^*\varepsilon}$). Now using (\ref{wboundv}), dividing by $\rho^{d-1}$ sending $\varepsilon \to 0$, $\eta \to 0$ and $\rho \to 0$ the definition of $\varphi(x_0,\nu_u(x_0))$ and $q(x_0)$ we obtain the claim.
\end{proof}
\begin{lemma}\label{LiminfBulkLemma}
\begin{align*}
g(x_0) \geq f(x_0,\nabla u(x_0)) \text{ for } \mathcal{L}^d \text{-a.e. } x \in \Omega.
\end{align*}
\end{lemma}
\begin{proof} Let $x_0 \in \Omega$ be a point of approximate differentiability of $u$. Set $u_0 = u(x_0) + \nabla u(x_0)(x-x_0)$. By \cite{ambrosio2000functions}, Theorem 3.83, we have that this property is satisfied for a.e. $x_0 \in \Omega$ and we have
\begin{itemize}
\item[a)]
$ \displaystyle
\lim_{\rho \to 0} \fint_{Q^\nu_\rho(x_0)} \frac{\left|u(x)-u_0(x) \right|}{\rho}\mathrm{d}x=0.
$
\item[b)] $g(x_0)=\displaystyle \lim_{\rho \to 0} \frac{\mu(Q^\nu_\rho(x_0))}{\rho^d} $.
\end{itemize}
It suffices to prove the inequality for points in $\Omega$ satisfying a) and b) with $g(x_0) < +\infty$. Fix such a $x_0 \in \Omega$ and a sequence $\rho \to 0$ such that $\mu(\partial Q^\nu_\rho(x_0))=0$. By the weak convergence of $\mu_n $ to $\mu$ we have that
\begin{align*}
g(x_0) &= \lim_{\rho \to 0} \frac{\mu(Q^\nu_\rho(x_0))}{\rho^d} = \lim_{\rho \to 0} \frac{1}{\rho^d}  \lim_{n \to \infty} \mu_n(Q^\nu_\rho(x_0)) \\&= \lim_{\rho \to 0 } \frac{1}{\rho^{d}} \lim_{\varepsilon \to 0}\sum_{\xi \in V} \sum_{i \in Z_\varepsilon( Q^\nu_\rho(x_0))} \varepsilon^d W^\varepsilon_{i,\xi}\left(D^\xi_\varepsilon u(i)\right).
\end{align*}
Since the limit exist and are finite we have that for $\rho $ and $\varepsilon $ small enough there holds
\begin{align}\label{Energybound}
F_\varepsilon(u_\varepsilon,Q^\nu_\rho(x_0))\leq C\rho^{d}.
\end{align}
Set
$\displaystyle  M_\rho(x_0) =\max_{x\in Q^\nu_\rho(x_0)}u_0(x)$ and $\displaystyle m_\rho(x_0) = \min_{x\in Q^\nu_\rho(x_0)}u_0(x)$. Furthermore we define
\begin{align*}
u_\varepsilon^\rho(x) = \left( u_\varepsilon(x) \vee M_\rho(x_0)\right) \wedge m_\rho(x_0).
\end{align*}
Note that $M_\rho(x_0) -m_\rho(x_0) \leq C\rho$,  $\left|u^\rho_\varepsilon(x)-u_0(x)\right| \leq C\rho$
and by using a) and $u_\varepsilon^\rho \to u $ in $L^1(\Omega)$ we have that
\begin{align}\label{int0}
\lim_{\rho \to 0}\lim_{\varepsilon \to 0} \fint_{Q^\nu_\rho(x_0)} \frac{\left|u^\rho_\varepsilon(x)-u_0(x)\right|}{\rho}\mathrm{d}x=0
\end{align}
and since truncation lowers the energy we have that
\begin{align*}
F_\varepsilon(u_\varepsilon^\rho,Q^\nu_\rho(x_0)) \leq F_\varepsilon(u_\varepsilon,Q^\nu_\rho(x_0)).
\end{align*}
Fix $K\in \mathbb{N}$, $\delta>0$ and define for $k \in \{ K,\ldots,2K-1\}$ a cut-off function $\varphi_k \in C^\infty(\mathbb{R}^d)$ between $ (Q^\nu_\rho(x_0))^+_{\frac{k+1}{K}\rho\delta} $ and $ (Q^\nu_\rho(x_0))^+_{\frac{k}{K}\rho\delta}$, i.e. 
\begin{align*}
||\nabla \varphi||_\infty \leq \frac{CK}{\rho\delta}, (Q^\nu_\rho(x_0))^+_{\frac{k+1}{K}\rho\delta} \subset\{\varphi=1\},\mathrm{supp}(\varphi_k) \subset (Q^\nu_\rho(x_0))^+_{\frac{k}{K}\rho\delta}.
\end{align*}
For $k \in \{K,\ldots,2K-1\}$ we define  $w_{\varepsilon,\rho}^k \in \mathcal{PC}_\varepsilon(\Omega)$ by
\begin{align*}
w_{\varepsilon,\rho}^k(i) = \varphi_k(i)u_\varepsilon^\rho(i) + (1-\varphi_k(i))u_0(i).
\end{align*}
Note that
\begin{align} \label{decompositiongradient}
D^\xi_\varepsilon w_{\varepsilon,\rho}^k(i)=\varphi_k(i) D^\xi_\varepsilon u^\rho_\varepsilon(i) + (1-\varphi_k(i))\nabla u(x_0)\xi + D^\xi_\varepsilon \varphi_k(i)(u^\rho_\varepsilon(i+\varepsilon\xi)-u_0(i))
\end{align}
and for $a,b \geq 0$ we have that
\begin{align}\label{Potentialsubadd}
\frac{1}{\varepsilon} \wedge (a+b)^2 \leq 2\left((\frac{1}{\varepsilon}\wedge a^2)+ (\frac{1}{\varepsilon}\wedge b^2)\right).
\end{align}
Setting $ \displaystyle R = \max_{\xi \in V} ||\xi||_\infty$, $S_{k,\varepsilon} = (Q^\nu_\rho(x_0))_{\frac{k+1}{K}\rho +R\varepsilon}^+\setminus  (Q^\nu_\rho(x_0))_{\frac{k}{K}\rho -R\varepsilon}^+ $, splitting the energy into three contributions, the set where $\varphi_k(i),\varphi(i+\xi)=1 $, the set where $\varphi_k(i),\varphi(i+\xi)=0$ and the set where neither holds true, we obtain
\begin{align*}
F_\varepsilon(w^k_{\varepsilon,\rho},Q^\nu_\rho(x_0)) \leq &\sum_{\xi \in V} \sum_{i \in  Z_\varepsilon\left( (Q^\nu_\rho(x_0))_{\frac{k+1}{K}\rho\delta +R\varepsilon}^+\right) } \varepsilon^d  W^\varepsilon_{i,\xi}\left(D^\xi_\varepsilon w^k_{\varepsilon,\rho}(i)\right) \\+& \sum_{\xi \in V} \sum_{i \in Z_\varepsilon( S_{k,\varepsilon})}\varepsilon^d W^\varepsilon_{i,\xi}\left(D^\xi_\varepsilon w^k_{\varepsilon,\rho}(i)\right) \\+& \sum_{\xi \in V} \sum_{i \in Z_\varepsilon\left((Q^\nu_{\rho}(x_0))_{\frac{k}{K}\rho\delta -R\varepsilon}\right)} \varepsilon^d W^\varepsilon_{i,\xi}\left(D^\xi_\varepsilon w^k_{\varepsilon,\rho}(i)\right).
\end{align*}
For the first term we have, since $w_{\varepsilon,\rho}^k = u^\rho_\varepsilon$ for all $i,i+\varepsilon\xi$ that appear in the sum,
\begin{align} \label{Bulkbound}
F_\varepsilon(w^k_{\varepsilon,\rho}, (Q^\nu_\rho(x_0))^+_{\frac{k+1}{K}\delta\rho +\varepsilon R}) \leq F_\varepsilon(u_\varepsilon^\rho,Q^\nu_\rho(x_0)).
\end{align}
For the second term we have, noting (\ref{decompositiongradient}), (\ref{Potentialsubadd}), the definition of $\varphi_k$, 
\begin{align*}
|D^\xi_\varepsilon w^k_{\varepsilon,\rho}(i)| \leq |D^\xi_\varepsilon u^\rho_\varepsilon(i)|+|D^\xi_\varepsilon u_0(i)| + |D^\xi_\varepsilon \varphi_k(i)| |u_\varepsilon^\rho(i+\xi) - u_0(i+\xi)|
\end{align*}
and $W^\varepsilon_{i,\xi}(z) \leq z^2$,
\begin{align}\label{Stripebound}
\begin{split}
F_\varepsilon(w^k_{\varepsilon,\rho},S_{k,\varepsilon}) \leq &C\left( F_\varepsilon(u^\rho_\varepsilon,S_{k,\varepsilon}) + F_\varepsilon(u_0(i),S_{k,\varepsilon})\right) \\ +& \frac{CK^2}{\rho^2\delta^2}\sum_{i\in Z_\varepsilon( Q^\nu_\rho(x_0))}\varepsilon^d |u_\varepsilon^\rho(i) - u_0(i)|^2.
\end{split}
\end{align}
By the definition of $u_\varepsilon^\rho$ we have that
\begin{align*}
 |u_\varepsilon^\rho(i) - u_0(i)|^2 \leq C\rho |u_\varepsilon^\rho(i) -u_0(i)|
\end{align*}
for all $i \in \varepsilon\mathbb{Z}^d \cap Q^\nu_\rho(x_0)$.
and therefore 
\begin{align}\label{intbound}
\frac{CK^2}{\rho^2\delta^2}\sum_{i\in Z_\varepsilon( Q^\nu_\rho(x_0))}\varepsilon^d |u_\varepsilon^\rho(i) - u_0(i)|^2 \leq C\frac{K^2}{\delta^2} \int_{Q^\nu_\rho(x_0)} \frac{\left|u^\rho_\varepsilon(x)-u_0(x) \right|}{\rho}\mathrm{d}x.
\end{align}
The third term can be estimated by
\begin{align*}
\sum_{\xi \in V} \sum_{i \in Z_\varepsilon\left((Q^\nu_{\rho}(x_0))_{\frac{k}{K}\rho -R\varepsilon}\right)} \varepsilon^d W^\varepsilon_{i,\xi}\left(D^\xi_\varepsilon w^k_{\varepsilon,\rho}(i)\right) \leq \sum_{\xi \in V} \sum_{i \in Z_\varepsilon\left( (Q^\nu_{\rho}(x_0))_{\frac{k}{K}\delta\rho -R\varepsilon}\right)}\varepsilon^d  |D^\xi_\varepsilon u_0(i)|^2  \leq C\delta\rho^{d-1}
\end{align*}
Note that for $\varepsilon >0$ small enough $S_{k,\varepsilon} \cap S_{j,\varepsilon}=\emptyset$ for all $|k-j| \geq 2$ and therefore, averaging over $k \in \{K,\ldots 2K-1\}$, we obtain
\begin{align}\label{averagebound}
\begin{split}
\frac{1}{K}\sum_{k=1}^K F_\varepsilon(w^k_{\varepsilon,\rho},Q^\nu_\rho(x_0)) &\leq F_\varepsilon(u_\varepsilon^\rho,Q^\nu_\rho(x_0)) + \frac{C}{K}\left( F_\varepsilon(u^\rho_\varepsilon,Q^\nu_\rho(x_0)) + F_\varepsilon(u_0(i),Q^\nu_\rho(x_0))\right) \\&+ C\frac{K}{\delta^2} \int_{Q^\nu_\rho(x_0)} \frac{\left|u^\rho_\varepsilon(x)-u_0(x) \right|}{\rho}\mathrm{d}x + o(\rho^{d})\\&\leq  F_\varepsilon(u_\varepsilon^\rho,Q^\nu_\rho(x_0)) +  \frac{C}{K}\rho^d + C\left(\frac{K}{\delta^2}+1\right) o(\rho^d)
\end{split}
\end{align}
where we used (\ref{Energybound}), (\ref{int0}) and (\ref{Bulkbound})-(\ref{intbound}). Now choosing $k(\varepsilon) \in \{K,\ldots2K-1\}$ such that 
\begin{align*}
 F_\varepsilon(w^{k(\varepsilon)}_{\varepsilon,\rho},Q^\nu_\rho(x_0))\leq\frac{1}{K}\sum_{k=1}^K F_\varepsilon(w^k_{\varepsilon,\rho},Q^\nu_\rho(x_0)).
\end{align*} 
Now dividing by $\rho^d$, sending $\varepsilon \to 0$,$\rho \to 0$ and $K \to \infty$ we have that
\begin{align}\label{limineq}
\lim_{K\to \infty}\lim_{\rho \to 0} \frac{1}{\rho^d}\lim_{\varepsilon\to 0}F_\varepsilon(w^{k(\varepsilon)}_{\varepsilon,\rho},Q^\nu_\rho(x_0)) \leq \lim_{\rho \to 0} \frac{1}{\rho^d} \lim_{\varepsilon \to 0} F_\varepsilon(u_\varepsilon,Q^\nu_\rho(x_0)) = g(x_0).
\end{align}
Now note that for fixed $\varepsilon,\rho>0$ and $k \in \{K,\ldots,2K-1\}$ we have that if $\eta < \rho\delta$ it holds $w^{k}_{\varepsilon,\rho}(i)= u_0(i) $ for all $i \in Z_\varepsilon\left( (Q^\nu_\rho(x_0))_\eta\right)$ and therefore, by noting that
\begin{align*}
F_\varepsilon(v+c,Q^\rho_\nu(x)) = F_\varepsilon(v,Q^\rho_\nu(x))
\end{align*} 
for all $c \in \mathbb{R}$, we have
\begin{align*}
m_{\varepsilon,\eta}^{F_\varepsilon}(\nabla u(x_0)\cdot,Q_\rho^\nu(x_0))\leq F_\varepsilon(w^{k(\varepsilon)}_{\varepsilon,\rho},Q^\nu_\rho(x_0)).
\end{align*}
Noting (\ref{limineq}) the claim follows.
\end{proof}

Now we introduce some notation in order to prove the limsup inequality. This is done in two steps - In the first step we use a density argument to reduce to a smooth class of functions (defined in the following) and in the second step we use the cell-formulas to construct a recovery sequence for that class.

\medskip

Let $R\subset\subset \Omega$ be a $(d-1)$-dimensional compact $C^1$ manifold with $C^1$ boundary. For $\mathcal{H}^{d-1}$-a.e. $x_0 \in R$, $\rho >0, \nu = \nu_R(x_0) \in S^{d-1}$ there exists $f: \mathbb{R}^{d-1} \to \mathbb{R}  $ such that after rotation, writing $x = (x^{\prime},x_d)$, we have
\begin{align}\label{manifold}
R \cap &Q_\rho(x_0) \subset \{(x^{\prime},x_d) \in Q_\rho(x_0) : x_d = f(x^{\prime}) \}.
\end{align}
and we set
\begin{align} \label{Cubepm}
Q^\pm_\rho(x_0)= \{(x^{\prime},x_d) \in Q_\rho(x_0) : \pm x_d > f(x^{\prime}) \} \subset \Omega.
\end{align}
The functions in $\mathcal{D}_2(\Omega)$, which we prove to be dense, are functions that except for a finite union of $C^1$-manifolds $M$ are smooth up to the boundary of $\Omega\setminus M$ and may only jump along $M$. We strongly make use of \cite{de2016approximation}, where the main approximation result, that we use, is stated. However if we localize at a point $x_0\in S(u)$ we further need the property that our functions are $C^\infty$ up to the boundary of $Q^\pm_\rho(x_0)$.
We define $\mathcal{D}_2(\Omega) \subset SBV^2(\Omega)$ by
\begin{align} \label{D2def}
\mathcal{D}_2(\Omega)= \Bigg\{&u \in SBV^2(\Omega) \cap L^\infty(\Omega) : \,\exists \,  M \text{ finite union of compact } C^1\text{-manifolds with}\\\nonumber &C^1\text{-boundary}, M \subset\subset \Omega,  S(u) \subset M, \mathcal{H}^{d-1}(M\setminus S(u))=0, u \in C^\infty( \Omega \setminus  M) \\&\nonumber\text{ and for } \mathcal{H}^{d-1}\text{-a.e. } x \in S(u) \text{ and } \rho >0 \text{ small enough }  u_j \in C^\infty(\overline{Q_\rho^\pm(x_0)}) \Bigg\}
\end{align}

\begin{lemma}[Approximation Lemma]\label{ApproximationLemma} Let $\Omega \subset \mathbb{R}^d$ be an open, bounded and Lipschitz set and let $u \in SBV^2(\Omega)\cap L^\infty(\Omega)$. Then there exists a sequence of functions $\{u_j\}_j \subset \mathcal{D}_2(\Omega)$ such that
\begin{align*}
||u_j-u||_{BV(\Omega)} \to 0,\quad  \nabla u_j \to \nabla u \text{ in } L^2(\Omega;\mathbb{R}^d), \quad \mathcal{H}^{d-1}(S(u_j)\triangle S(u)) \to 0.
\end{align*}
\end{lemma}
\begin{proof}  By \cite{de2016approximation} Theorem C it suffices to prove the claim for $u \in SBV^2(\Omega) \cap L^\infty(\Omega)$ such that there exists a $C^1$-manifold $M$ with (possibly empty) $C^1$-boundary, $M\subset\subset \Omega$, such that $J(u) \subset M$, $\mathcal{H}^{d-1}(M \setminus J(u))=0$ and $u \in C^\infty(\Omega \setminus \overline{J(u)})$. Let $\delta>0$ and let $\varepsilon>0$ be such that $\partial (\partial M)_\varepsilon$ is a $C^1$-manifold  such that $\mathcal{H}^{d-1}(M \cap \partial (\partial M)_\varepsilon)=0$ and
\begin{align} \label{epschosen}
(2||u||_\infty+1) (\mathcal{H}^{d-1}(\partial(\partial M)_\varepsilon) + \mathcal{H}^{d-1}(M\cap (\partial M)_\varepsilon) < \delta, ||u||_{W^{1,2}((\partial M)_\varepsilon)} <\delta.
\end{align}
Let 
\begin{align*}
\mathcal{Q} = \Big\{ Q\subset \Omega : Q=Q^\nu_\rho(x_0), x_0 \in M, \nu = \nu_M(x_0), \rho >0 \text{ and } (\ref{manifold}) \text{ is satisfied}\Big\}.
\end{align*}
Since $M \setminus (\partial M)_{\frac{\varepsilon}{2}}$ is compact, there exists $\{\Omega_1^\varepsilon,\ldots,\Omega_{N_\varepsilon}^\varepsilon\}$ such that $\Omega_n^\varepsilon \in \mathcal{Q}$ for all $n \in\{1,\ldots,N_\varepsilon\}$ (i.e. $\Omega_n^\varepsilon = Q^{\nu_n}_{\rho_n}(x_n)$, with properties as above) and
\begin{align*}
M \setminus (\partial M)_{\frac{\varepsilon}{2}} \subset \bigcup_{n=1}^{N_\varepsilon} \Omega_n^\varepsilon.
\end{align*}
Let $\mathrm{d}_\varepsilon =\min \{\{\rho_n\}_{n=1}^{N_\varepsilon},\varepsilon\}$, set 
$\Omega_0^\varepsilon = \{x \in \Omega : \mathrm{dist}(x,M) > \frac{\mathrm{d}_\varepsilon}{2}\}$ be such that $\Omega_0^\varepsilon$ is a set with Lipschitz boundary and $\Omega_{N_\varepsilon+1}^\varepsilon = (\partial M)_\varepsilon$ . We have
\begin{align*}
\Omega \subset \bigcup_{n=0}^{N_\varepsilon+1} \Omega_n^\varepsilon
\end{align*}
and therefore there exists a partition of unity $\{\varphi_{n,\varepsilon}\}_{n=0}^{N_\varepsilon+1}$, i.e. 
\begin{align*}
\varphi_{n,\varepsilon} \in C_c^\infty(\Omega_n^\varepsilon),\quad 0\leq \varphi_{n,\varepsilon}\leq  1\text{ and } \sum_{n=0}^{N_\varepsilon+1} \varphi_{n,\varepsilon}=1.
\end{align*}
Since $\Omega_0^\varepsilon$ is a set with Lipschitz boundary and $|D^s(\varphi_{0,\varepsilon} u)|(\Omega_0^\varepsilon)=0$ we have that $\varphi_{0,\varepsilon} u \in W^{1,2}(\Omega_0^\varepsilon)$ and therefore there exists $\{u_{j,0}^\varepsilon\}_j \subset W^{1,2}(\Omega_0^\varepsilon) \cap C^\infty(\overline{\Omega_0^\varepsilon})$ such that $||u_{j,0}^\varepsilon ||_\infty \leq  ||u||_\infty$ and $u_{j,0}^\varepsilon \to \varphi_{0,\varepsilon}u$ strongly in $W^{1,2}(\Omega_0^\varepsilon)$. Now let $n \in \{1,\ldots,N_\varepsilon\}$. By property (\ref{manifold}) we have that 
\begin{align*}
\Omega_n^\varepsilon  =\Omega_n^{\varepsilon,+}\cup \Omega_n^{\varepsilon,-} \cup (\Omega_n^\varepsilon \cap M),
\end{align*}
where there exists $R_n \in SO(d)$ such that 
\begin{align*}
\Omega_n^{\varepsilon,\pm}&= R_n\{(x^{\prime},x_d) \in Q_\rho(x_0) : x_d \gtrless f(x^{\prime}) \}, \\(\Omega_n^\varepsilon \cap M) &= R_n\{(x^{\prime},x_d) \in Q_\rho(x_0) :  x_d = f(x^{\prime}) \}.
\end{align*}
Note that $\Omega_n^{\varepsilon,\pm}$ has a Lipschitz boundary and $|D^s( \varphi_{n,\varepsilon} u)|(\Omega_n^{\varepsilon,\pm})=0$. Therefore $\varphi_{n,\varepsilon}u \in W^{1,2}(\Omega_n^{\varepsilon,\pm})$ and we have that there exists $\{u_{j,n}^{\varepsilon,\pm}\}_j \subset W^{1,2}(\Omega_n^{\varepsilon,\pm}) \cap C^\infty(\overline{\Omega_n^{\varepsilon,\pm}})$ such that $||u_{j,n}^{\varepsilon,\pm} ||_\infty \leq  ||u||_\infty$ and $u_{j,n}^{\varepsilon,\pm} \to \varphi_{n,\varepsilon}u$ strongly in $W^{1,2}(\Omega_n^{\varepsilon,\pm})$. We define $u_{j,n}^{\varepsilon} \in SBV^2(\Omega_n^\varepsilon) \cap L^\infty(\Omega)$ by
\begin{align*}
u_{j,n}^{\varepsilon}(x) = \begin{cases}
u_{j,n}^{\varepsilon,+}(x) &x \in \Omega_n^{\varepsilon,+}\\
u_{j,n}^{\varepsilon,-}(x) &x \in \Omega_n^{\varepsilon,-}.
\end{cases}
\end{align*}
Now $S(u_{j,n}^\varepsilon) \subset M$ and $u_{j,n}^\varepsilon \in C^\infty(\overline{\Omega_n^{\varepsilon,\pm}})$. We define $u_{j}^\varepsilon \in SBV^2(\Omega) \cap L^\infty(\Omega)$ by
\begin{align*}
u_j^\varepsilon(x) = \begin{cases} \displaystyle \sum_{n=0}^{N_\varepsilon} u_{j,n}^\varepsilon(x) &x \in \Omega \setminus (\partial M)_\varepsilon \\
0 &\text{otherwise.}
\end{cases}
\end{align*}
Since $\mathrm{supp}(\varphi_{N_\varepsilon+1,\varepsilon}) \subset (\partial M)_\varepsilon$ and therefore $\sum_{n=1}^{N_\varepsilon} \varphi_{n,\varepsilon}(x)=1$ for all $x \in \Omega \setminus (\partial M)_\varepsilon$ we have that $u_{j}^\varepsilon \to u$ in $W^{1,2}(\Omega \setminus ((\partial M)_\varepsilon \cup M))$, therefore for $j $ big enough, using (\ref{epschosen}),
\begin{align*}
||u - u^\varepsilon_j||_{L^1(\Omega)} \leq C||u - u^\varepsilon_j||_{L^2(\Omega)} \leq C||u||_{L^2((\partial M)_\varepsilon)} + C||u - u^\varepsilon_j||_{L^1(\Omega \setminus (\partial M)_\varepsilon)} < C\delta
\end{align*}
and 
\begin{align*}
||\nabla u - \nabla u^\varepsilon_j||_{L^2(\Omega;\mathbb{R}^d)} \leq ||\nabla u||_{L^2((\partial M)_\varepsilon;\mathbb{R}^d)} +||u - u^\varepsilon_j||_{L^2(\Omega \setminus (\partial M)_\varepsilon;\mathbb{R}^d)} < C\delta.
\end{align*}
By H\"older's Inequality we also have that $||\nabla u - \nabla u^\varepsilon_j||_{L^1(\Omega;\mathbb{R}^d)} < C\delta$. Note that $S(u_j^\varepsilon) \subset (M \setminus (\partial M)_\varepsilon) \cup \partial (\partial M)_\varepsilon $. Since we have the strong $W^{1,2}(\Omega \setminus ((\partial M)_\varepsilon \cup M))$-convergence of $u_\varepsilon^j \to u$ we can apply locally the trace theorem and we have that $(u_j^{\varepsilon})^\pm \to u^\pm$ in $L^2(M \setminus (\partial M)_\varepsilon)$ and therefore, since $\mathcal{H}^{d-1}(M) < +\infty$ in $L^1(M \setminus (\partial M)_\varepsilon)$. By the same reasoning we also have $(u_j^{\varepsilon})^+ \to u^+$ in $L^1(\partial (\partial M)_\varepsilon)$.  We therefore have 
\begin{align*}
|D^s( u^\varepsilon_j -u)|(\Omega) &= \int_{S(u^\varepsilon_j-u)} |(u_j^{\varepsilon}-u)^+- (u_j^{\varepsilon}-u)^-|\mathrm{d}\mathcal{H}^{d-1} \\&\leq \int_{M\setminus (\partial M)_\varepsilon} |(u_j^{\varepsilon}-u)^+|+| (u_j^{\varepsilon}-u)^-|\mathrm{d}\mathcal{H}^{d-1}  \\&\quad +\int_{M\cap (\partial M)_\varepsilon} |u^+-u^-|\mathrm{d}\mathcal{H}^{d-1}  + \int_{\partial(\partial M)_\varepsilon} |(u_j^{\varepsilon})^+-(u_j^{\varepsilon})^-|\mathrm{d}\mathcal{H}^{d-1}
\end{align*}
Now for the first term we have that for $j $ big enough there holds
\begin{align}\label{Dsu1}
\begin{split}
\int_{M\setminus (\partial M)_\varepsilon} |(u_j^{\varepsilon}-u)^+|+| (u_j^{\varepsilon}-u)^-|\mathrm{d}\mathcal{H}^{d-1} = &||(u_j^{\varepsilon})^+-u^+||_{L^1(M\setminus (\partial M)_\varepsilon)} \\ + &||(u_j^{\varepsilon})^--u^-||_{L^1(M\setminus (\partial M)_\varepsilon)} < \delta.
\end{split}
\end{align}
For the second term we have by (\ref{epschosen})
\begin{align}\label{Dsu2}
\begin{split}
\int_{M\cap (\partial M)_\varepsilon} |u^+-u^-|\mathrm{d}\mathcal{H}^{d-1} \leq 2||u||_\infty \mathcal{H}^{d-1}(M\cap (\partial M)_\varepsilon) <\delta
\end{split}
\end{align}
whereas for the last term we have for $j$ big enough
\begin{align}\label{Dsu3}
\begin{split}
\int_{\partial(\partial M)_\varepsilon} |(u_j^{\varepsilon})^+-(u_j^{\varepsilon})^-|\mathrm{d}\mathcal{H}^{d-1} &=  ||(u^{\varepsilon}_j)^+||_{L^1(\partial(\partial M)_\varepsilon)} \leq ||u^+||_{L^1(\partial(\partial M)_\varepsilon)} \\&\leq 2||u||_\infty \mathcal{H}^{d-1}(\partial(\partial M)_\varepsilon) < \delta.
\end{split}
\end{align}
Hence for $j$ big enough we have that
$
|D^s( u^\varepsilon_j -u)|(\Omega) < C\delta.
$ Now
\begin{align*}
\mathcal{H}^{d-1}(S(u^\varepsilon_j) \triangle S(u)) &\leq  \mathcal{H}^{d-1}(S(u^\varepsilon_j) \setminus S(u)) + \mathcal{H}^{d-1}(S(u)\setminus S(u^\varepsilon_j) ) \\&\leq \mathcal{H}^{d-1}(\partial (\partial M)_\varepsilon)) + \mathcal{H}^{d-1}(M\cap (\partial M)_\varepsilon)) \\&\quad+ \mathcal{H}^{d-1}(( M \setminus  \overline{(\partial M)_\varepsilon})\setminus  S(u^\varepsilon_j) ) \\&\leq C\delta + \mathcal{H}^{d-1}((M \setminus \overline{(\partial M)_\varepsilon})\setminus  S(u^\varepsilon_j) ).
\end{align*}
Now since $u_j^\varepsilon \to u$ in $L^1(\Omega \setminus (\partial M)_\varepsilon)$, $\displaystyle\sup_{j}(|Du^\varepsilon_j|(\Omega \setminus (\partial M)_\varepsilon) + ||\nabla u_j^\varepsilon||_{L^2(\Omega \setminus (\partial M)_\varepsilon)}) <+\infty$,  and $S(u_j^\varepsilon) \subset (M \setminus (\partial M)_\varepsilon) \cup \partial (\partial M)_\varepsilon $ we have that
\begin{align*}
\mathcal{H}^{d-1}(S(u) \cap \Omega\setminus \overline{(\partial M)_\varepsilon})&=\mathcal{H}^{d-1}(M\setminus \overline{(\partial M)_\varepsilon})\geq\limsup_{j \to \infty} \mathcal{H}^{d-1}(S(u_j^\varepsilon) \cap \Omega \setminus \overline{(\partial M)_\varepsilon})\\&\geq \liminf_{j \to \infty} \mathcal{H}^{d-1}(S(u_j^\varepsilon) \cap \Omega \setminus \overline{(\partial M)_\varepsilon}) \\&\geq \mathcal{H}^{d-1}(S(u) \cap \Omega \setminus \overline{(\partial M)_\varepsilon}).
\end{align*}
Hence, noting $S(u_j^\varepsilon) \cap (\Omega \setminus \overline{(\partial M)_\varepsilon}) \subset S(u) \cap (\Omega \setminus \overline{(\partial M)_\varepsilon}) $ we have for $j$ big enough that
\begin{align*}
\mathcal{H}^{d-1}((M \setminus\overline{ (\partial M)_\varepsilon})\setminus  S(u^\varepsilon_j) ) &= \mathcal{H}^{d-1}((S(u) \setminus S(u^\varepsilon_j)) \cap (\Omega \setminus \overline{(\partial M)_\varepsilon})) \\&\leq \mathcal{H}^{d-1}(S(u) \setminus (\Omega\setminus \overline{(\partial M)_\varepsilon})) - \mathcal{H}^{d-1}(S(u_j^\varepsilon) \setminus (\Omega\setminus \overline{(\partial M)_\varepsilon})) \\& <\delta.
\end{align*}
It remains to prove that $u_j^\varepsilon \in \mathcal{D}_2(\Omega)$. By construction $u_j^\varepsilon \in SBV^2(\Omega) \cap L^\infty(\Omega)$. And setting as the union of manifolds in the definition of $\mathcal{D}_2(\Omega)$ the finite union of compact manifolds given by $K^j_\varepsilon = (M \cup \partial (\partial M)_\varepsilon) \setminus \{[u_j^\varepsilon] =0\} \subset\subset \Omega$ we see that $u_j^\varepsilon \in C^\infty(\Omega \setminus K_\varepsilon^j)$, $S(u_\varepsilon^j) \subset K_\varepsilon^j$, $\mathcal{H}^{d-1}(K_\varepsilon^j \setminus S(u_j^\varepsilon))=0$ and for $\mathcal{H}^{d-1}$-a.e. $ x \in S(u_j^\varepsilon)$ and $ \rho >0$  small enough $  u_j^\varepsilon \in C^\infty(\overline{Q_\rho^\pm(x_0)}) $. The claim follows by letting first $j \to \infty$ and then $\varepsilon \to 0$. 
\end{proof}
Before we prove Proposition \ref{GammalimsupLemma} we state some useful Lemmas and Propositions that will be used in the demonstration of it. We postpone their proves until after the proof of Proposition \ref{GammalimsupLemma}.

\medskip

Proposition \ref{Neihbourhoodcardlemma} will be used to estimate the cardinality of the points close to $M$.

\begin{proposition}\label{Neihbourhoodcardlemma} Let $\rho,\varepsilon >0$ and let $E \in\mathcal{B}(\mathbb{R}^d)$. Then
\begin{align*}
\#\left( E_\rho \cap \varepsilon\mathbb{Z}^d\right) \leq C\frac{|E_\rho|}{(\varepsilon\wedge \rho)^d}
\end{align*}
where $C$ depends only on $d$.
 \end{proposition}
 
The next two lemmas allows to construct the recovery sequence for a function $u \in \mathcal{D}_2(\Omega)$ for cubes which do not intersect the jump set and for cubes which do intersect the jump set respectively.
 
 \begin{lemma}\label{EqualityBulkLemma} Let $u \in \mathcal{D}_2(\Omega)$. For $\mathcal{L}^d$-a.e. $x \in \Omega$ and all $\nu \in S^{d-1}$ there holds
\begin{align*}
\lim_{\rho \to 0} \frac{1}{\rho^d}\lim_{\eta \to 0} \lim_{\varepsilon \to 0} m_\varepsilon^\eta(u,Q^\nu_\rho(x_0)) = \lim_{\rho \to 0} \frac{1}{\rho^d}\lim_{\eta \to 0} \lim_{\varepsilon \to 0} m_\varepsilon^\eta(u(x_0)+\nabla u(x_0)(\cdot-x_0),Q^\nu_\rho(x_0)).
\end{align*}
\end{lemma}
\begin{lemma}\label{EqualitySurfaceLemma} Let $u \in \mathcal{D}_2(\Omega)$. For $\mathcal{H}^{d-1}$-a.e. $x \in S(u)$ there holds
\begin{align*}
\lim_{\rho \to 0} \frac{1}{\rho^{d-1}}\lim_{\eta \to 0} \lim_{\varepsilon \to 0} m_\varepsilon^{\frac{\eta}{2}}(u,Q^\nu_\rho(x_0)) \leq \lim_{\rho \to 0} \frac{1}{\rho^{d-1}}\lim_{\eta \to 0} \lim_{\varepsilon \to 0} m_\varepsilon^\eta(u_{x_0,\nu},Q^\nu_\rho(x_0)).
\end{align*}
\end{lemma}

Finally we use a coarser estimate for points, that do not lie inside one of the cubes of the covering, since for those it suffices to know that their contribution is negligible for sufficiently smooth whose measure tends to zero.

\begin{lemma}\label{UpperboundDxi} Let $u \in \mathcal{D}_2(\Omega)$. Let $i \in \varepsilon\mathbb{Z}^d \cap \Omega$ be such that $\mathrm{dist}([i,i+\varepsilon\xi], M) > \varepsilon $, then
\begin{align*}
 \varepsilon^d |D^\xi_\varepsilon u(i)|^2 \leq C\int_{([i,i+\varepsilon\xi])_\varepsilon}|\nabla u(x)|^2\mathrm{d}x
\end{align*}
for some constant $C$ depending only $d$.
\end{lemma}

\begin{proposition} \label{GammalimsupLemma}
\begin{align*}
F^{\prime\prime}(u) \leq F(u).
\end{align*}
\end{proposition}
\begin{proof} Since truncation lowers the energy we can assume that $u \in SBV^2(\Omega) \cap L^\infty(\Omega)$. By Lemma \ref{ApproximationLemma} we can assume that $u \in \mathcal{D}_2(\Omega)$, i.e. there exists a finite union of compact $C^1$-manifolds with (possibly empty) $C^1$-boundaries such that $S(u) \subset M$, $\mathcal{H}^{d-1}(M\setminus S(u))=0$, $u \in C^\infty(\Omega\setminus \overline{M})$ and for $ \mathcal{H}^{d-1}$-a.e. $ x \in S(u)$  and  $\rho >0$  small enough   $u_j \in C^\infty(\overline{Q_\rho^\pm(x_0)})$. In fact assume, that we proved Lemma \ref{GammalimsupLemma} for such functions, then we know, that for every $u \in SBV^2(\Omega)$ there exists a sequence $\{u_j\}_j$ converging to $u$ with respect to the $L^1(\Omega)$-topology, satisfying the above properties and such that $\mathcal{H}^{d-1}(J(u) \triangle J(u_j)) \to 0$ and $\nabla u_j \to \nabla u $ with respect to the strong $L^2(\Omega;\mathbb{R}^d)$ topology. Since $f(x,\cdot)$ is quasiconvex, it is locally lipschitz continuous and it satisfies (\ref{fbounds}). Therefore
\begin{align*}
\lim_{j \to \infty} \int_\Omega f(x,\nabla u_j(x))\mathrm{d}x = \int_\Omega f(x,\nabla u(x))\mathrm{d}x. 
\end{align*}
Since $\mathcal{H}^{d-1}(J(u) \triangle J(u_j)) \to 0$, noting $\mathcal{H}^{d-1}(S(u) \setminus J(u))=0$ for all $u \in BV(\Omega)$, we have that
\begin{align*}
\lim_{j \to 0} \int_{S(u_j)}\varphi(x,\nu_{u_j}(x))\mathrm{d}\mathcal{H}^{d-1} = \int_{S(u)}\varphi(x,\nu_{u}(x))\mathrm{d}\mathcal{H}^{d-1}. 
\end{align*}
Therefore by the lower-semicontinuity of $F^{\prime\prime}$ with respect to the strong $L^1(\Omega)$-topology we have
\begin{align*}
F^{\prime\prime}(u) \leq \liminf_{j \to \infty} F^{\prime\prime}(u_j) &= \liminf_{j \to \infty}\left( \int_\Omega f(x,\nabla u_j(x))\mathrm{d}x + \int_{S(u_j)}\varphi(x,\nu_{u_j}(x))\mathrm{d}\mathcal{H}^{d-1} \right) \\& = \int_\Omega f(x,\nabla u(x))\mathrm{d}x +  \int_{S(u)}\varphi(x,\nu_{u}(x))\mathrm{d}\mathcal{H}^{d-1}=F(u)
\end{align*}
which yields the claim. It remains to prove (\ref{GammalimsupLemma}) for $u \in \mathcal{D}_2(\Omega)$. To this end let $\delta>0$ and define the set of all finite families of closed cubes satisfying (i)-(iv)
\begin{align*}
\mathcal{Q}_\delta =\{ Q_j\}_{j=1}^N
\end{align*}
\begin{itemize}
\item[(i)] $Q_j = Q^{\nu_j}_{\rho_j}(x_j)$, $\rho_j< \delta$ for all $j \in \{1,\ldots,N\}$, $Q_j \cap Q_k =\emptyset $ for all  $j,k \in \{1,\ldots,N\}$, $j \neq k$ For each $j \in \{1,\ldots,N\}$ we have either $x_j \in S(u)$ or $Q_j \cap S(u) = \emptyset$.
\item[(ii)] $\displaystyle \mathcal{L}^d(\Omega \setminus  \bigcup_{j=1}^{N} Q_j) <\delta$, $\mathcal{H}^{d-1}(S(u) \setminus  \bigcup_{j=1}^{N} Q_j) <\delta$, $\mathcal{H}^{d-1}(S(u) \cap \partial Q_j)=0$ and $\mathcal{H}^{d-1}(S(u) \cap Q_j) \geq \frac{1}{2}\rho_j^{d-1}$ for all $j \in \{1,\ldots,N\}$, where $S(u) \cap Q_j \neq \emptyset$.
\item[(iii)] If $x_j \in S(u)$, then $\nu_j = \nu_u(x_0)$, $\displaystyle \varphi(x_j,\nu_j) \leq\frac{1}{\rho^{d-1}}\int_{S(u)\cap Q_j}\varphi(x,\nu_u(x))\mathrm{d}\mathcal{H}^{d-1} + \delta $ and
\begin{align*}
\lim_{\eta \to 0} \lim_{\varepsilon \to 0} m_\varepsilon^{\frac{\eta}{2}}(u,Q_j) \leq \lim_{\eta \to 0} \lim_{\varepsilon \to 0} m_\varepsilon^\eta(u_{x_0,\nu},Q_j) + \delta \rho^{d-1} .
\end{align*}
\item[(iv)] If $x_j \notin S(u)$, then $\displaystyle f(x_0,\nabla u(x_0)) \leq \frac{1}{\rho^d} \int_{Q_j}f(x,\nabla u(x))\mathrm{d}x +\delta $ and 
\begin{align*}
\lim_{\eta \to 0} \lim_{\varepsilon \to 0} m_\varepsilon^\eta(u,Q_j) \leq \lim_{\eta \to 0} \lim_{\varepsilon \to 0} m_\varepsilon^\eta(u(x_0)+ \nabla u(x_0)(\cdot-x_0),Q_j) + \delta \rho^{d} .
\end{align*}
\end{itemize}
The existence of such a family is guaranteed by the Besicovitch-Covering Theorem and the fact that (iii)-(iv) hold for $\mathcal{H}^{d-1}$-a.e. $x \in S(u)$ and $\mathcal{L}^d$-a.e. $x \in \Omega$ for $\rho>0$ small enough (cf. \cite{ambrosio2000functions}, Lemma \ref{EqualityBulkLemma} and Lemma \ref{EqualitySurfaceLemma}). This is done by fixing $\eta>0$ such that $|S(u)_\eta| <\frac{\delta}{2}$ and applying once Besicovitch-Covering Theorem to the measure $\mu_1 = \mathcal{H}^{d-1}\lfloor_{S(u)}$ and the family of cubes
\begin{align*}
\mathcal{Q}_1 = \Bigg\{Q = Q^{\nu}_{\rho}(x_0), x_0 \in S(u), \nu = \nu_u(x_0), \rho_j< \delta \wedge \eta \text{ for all } j \in \{1,\ldots,N\}, \Bigg\}
\end{align*}
satisfying (iii) and $\mathcal{H}^{d-1}(S(u) \cap \partial Q_j)=0$ and $\mathcal{H}^{d-1}(S(u) \cap Q_j) \geq \frac{1}{2}\rho_j^{d-1}$ for all $j \in \{1,\ldots,N\}$. Applying Besicovitch-Covering Theorem one more time to the measure $\mu_2 = \mathcal{L}^{d}$ and the family of cubes
\begin{align*}
\mathcal{Q}_2 = \Bigg\{Q = Q^{\nu}_{\rho}(x_0), x_0 \in \Omega, \nu \in S^{d-1}, \rho_j< \delta, \text{ for all } j \in \{1,\ldots,N\},  Q\subset \Omega \setminus (S(u))_\eta\Bigg\}
\end{align*}
satisfying (iv) we get to disjoint finite families $\{Q_j^1\}_{j=1}^{N_1}, Q^1_j \subset (S(u))_\eta$ for all $j \in \{1,\ldots,N_1\}$ and  $\{Q_j^2\}_{j=1}^{N_2}, Q^2_j \subset  \Omega \setminus (S(u))_\eta$ for all $j \in \{1,\ldots,N_2\}$ such that
\begin{align*}
\mathcal{H}^{d-1}(S(u) \setminus  \bigcup_{j=1}^{N_1} Q^1_j) <\delta,  \mathcal{L}^d\left(\left(\Omega \setminus (S(u))_\eta\right)\setminus  \bigcup_{j=1}^{N_2} Q^2_j \right) < \frac{\delta}{2}.
\end{align*}
Now the family $\{Q_j\}_{j=1}^{N}=\{ Q^1_j\}_{j=1}^{N_1} \cup  \{Q^2_j\}_{j=1}^{N_2}$ satisfies (i)-(iv).
For each $j \in \{1,\ldots,N\} $ such that $x_j \in S(u)$ let $u_{\varepsilon,j}^{\eta,\delta} \in \mathcal{PC}_\varepsilon(\Omega)$ be such that $||u_{\varepsilon,j}^{\eta,\delta}||_\infty \leq ||u||_\infty $, $(u_{\varepsilon,j}^{\eta,\delta})_i = u_i $ on $(Q_j)_\eta$ and
\begin{align*}
F_\varepsilon(u_{\varepsilon,j}^{\eta,\delta},Q_j)\leq m^\eta_\varepsilon(u,Q_j) + \delta\rho_j^{d-1}
\end{align*}
and for each $ j \in \{1,\ldots,N\}$ such that $Q_j \cap S(u) =\emptyset $ let $u_{\varepsilon,j}^{\eta,\delta} \in \mathcal{PC}_\varepsilon(\Omega)$ be such that $||u_{\varepsilon,j}^{\eta,\delta}||_\infty \leq ||u||_\infty $, $(u_\varepsilon^{\eta,\delta})_i = u_i $ on $(Q_j)_\eta$ and
\begin{align*}
F_\varepsilon(u_{\varepsilon,j}^{\eta,\delta},Q_j)\leq m^\eta_\varepsilon(u,Q_j) + \delta\rho_j^{d}.
\end{align*}
Define $u_{\varepsilon}^{\delta,\eta} \in \mathcal{PC}_\varepsilon(\Omega) $ by
\begin{align*}
u_{\varepsilon}^{\delta,\eta}(i) =\begin{cases}u_{\varepsilon,j}^{\eta,\delta}(i) &i \in Z_\varepsilon( Q_j), j \in \{1,\ldots,N\}\\
u(i)&\text{otherwise}.
\end{cases}
\end{align*} 
Note that
\begin{align*}
\mathrm{dist}(Q_j,Q_k) >d_\delta >0 \text{ for all } j \neq k,
\end{align*}
since $Q_j$ is a finite family of closed cubes such that $Q_j \cap Q_k = \emptyset$ for all $j \neq k$ and therefore
\begin{align*}
F_\varepsilon(u_{\varepsilon}^{\delta,\eta},Q_j) = F_\varepsilon(u_{\varepsilon,j}^{\eta,\delta},Q_j)
\end{align*}
for $\varepsilon >0$ small enough. Hence we have that
\begin{align*}
F_\varepsilon(u_{\varepsilon}^{\delta,\eta})& = \sum_{j =1}^N F_\varepsilon(u_{\varepsilon}^{\delta,\eta},Q_j) + F_\varepsilon(u_{\varepsilon}^{\delta,\eta}, \Omega \setminus \bigcup_{j=1}^N Q_j) \\&= \underset{x_j \in S(u)}{\sum_{j=1}^N} F_\varepsilon(u_{\varepsilon}^{\delta,\eta},Q_j) + \underset{x_j \notin S(u)}{\sum_{j=1}^N} F_\varepsilon(u_{\varepsilon}^{\delta,\eta},Q_j) + F_\varepsilon(u_{\varepsilon}^{\delta,\eta}, \Omega \setminus \bigcup_{j=1}^N Q_j) \\& \leq \underset{x_j \in S(u)}{\sum_{j=1}^N }( m^\eta_\varepsilon(u,Q_j)+\delta \rho_j^{d-1} ) + \underset{x_j \notin S(u)}{\sum_{j=1}^N }( m^\eta_\varepsilon(u,Q_j) +\delta \rho_j^{d} ) + F_\varepsilon(u_{\varepsilon}^{\delta,\eta}, \Omega \setminus \bigcup_{j=1}^N Q_j).
\end{align*}
Now note that by (ii) and (iii) and using the definition of $\varphi(x_j,\nu_j)$ we have for $\varepsilon,\eta>0$ small enough
\begin{align}\label{SurfaceLimsup}
\begin{split}
\underset{x_j \in S(u)}{\sum_{j=1}^N }( m^\eta_\varepsilon(u,Q_j)+\delta \rho_j^{d-1} ) &\leq \underset{x_j \in S(u)}{\sum_{j=1}^N}\left( \rho_j^{d-1}\varphi(x_j,\nu_j) + C\delta\rho_j^{d-1}\right) \\&\leq \underset{x_j \in S(u)}{\sum_{j=1}^N} \left(\int_{S(u)\cap Q_j} \varphi(x,\nu_u(x))\mathrm{d}\mathcal{H}^{d-1} +C\delta\rho_j^{d-1} \right)\\&\leq \int_{S(u)}\varphi(x,\nu_u(x))\mathrm{d}\mathcal{H}^{d-1} + C\delta\mathcal{H}^{d-1}(S(u)).
\end{split}
\end{align}
Now by (iv) we have
\begin{align}\label{BulkLimsup}
\begin{split}
\underset{x_j \notin S(u)}{\sum_{j=1}^N }( m^\eta_\varepsilon(u,Q_j)+\delta \rho_j^{d} ) &\leq \underset{x_j \notin S(u)}{\sum_{j=1}^N}\left( \rho_j^{d}\varphi(x_j,\nu_j) + C\delta\rho_j^{d}\right) \\&\leq \underset{x_j \notin S(u)}{\sum_{j=1}^N} \left(\int_{ Q_j} f(x,\nabla u(x))\mathrm{d}x +C\delta\rho_j^{d} \right)\\&\leq \int_{\Omega}f(x,\nabla u(x))\mathrm{d}x + C\delta |\Omega|.
\end{split}
\end{align}
Now the last term can be estimated by splitting into points which are close to $S(u)$ (which is well behaved, since it is contained in $M$ and $\mathcal{H}^{d-1}(M\setminus S(u))=0$) and points which are far away. More precisely, setting $\displaystyle R = \sup_{\xi \in V} |\xi|$, we have
\begin{align*}
F_\varepsilon(u_{\varepsilon}^{\delta,\eta}, \Omega \setminus \bigcup_{j=1}^N Q_j) = &\sum_{\xi \in V} \underset{\mathrm{dist}(i,S(u))> 2R\varepsilon}{\sum_{ i \in Z_\varepsilon\left(\Omega \setminus \bigcup_{j=1}^N Q_j\right) }} \varepsilon^d W^\varepsilon_{i,\xi}(D^\xi_\varepsilon u_{\varepsilon}^{\delta,\eta}(i)) \\&+ \sum_{\xi \in V} \underset{\mathrm{dist}(i,S(u))\leq  2R\varepsilon}{\sum_{ i \in Z_\varepsilon\left(\Omega \setminus \bigcup_{j=1}^N Q_j\right) }} \varepsilon^d W^\varepsilon_{i,\xi}(D^\xi_\varepsilon u_{\varepsilon}^{\delta,\eta}(i))
\end{align*}
Now, using Lemma \ref{UpperboundDxi} and noting $D^\xi_\varepsilon u_{\varepsilon}^{\delta,\eta} = D^\xi_\varepsilon u$ on that set, we have
\begin{align}\label{Bulkpartfuori}
\begin{split}
\sum_{\xi \in V} \underset{\mathrm{dist}(i,S(u))> 2R\varepsilon}{\sum_{ i \in Z_\varepsilon\left( \Omega \setminus \bigcup_{j=1}^N Q_j \right)}} \varepsilon^d W^\varepsilon_{i,\xi}(D^\xi_\varepsilon u_{\varepsilon}^{\delta,\eta}(i)) &\leq \sum_{\xi \in V} \underset{\mathrm{dist}(i,S(u))> 2R\varepsilon}{\sum_{ i \in Z_\varepsilon\left( \Omega \setminus \bigcup_{j=1}^N Q_j \right)}} \varepsilon^d |D^\xi_\varepsilon u(i)|^2 \\&\leq C\int_{(\Omega \setminus \bigcup_{j=1}^N Q_j)_{(2R+1)\varepsilon}}|\nabla u(x)|^2 \mathrm{d}x.
\end{split}
\end{align}
Whereas, using Lemma \ref{Neihbourhoodcardlemma}, (ii) and the fact that $|(M)_{2R\varepsilon} \setminus \bigcup_{j=1}^N Q_j | \leq C\varepsilon\mathcal{H}^{d-1}(M \setminus \bigcup_{j=1}^N Q_j) +C(u)\varepsilon^2$, we obtain
\begin{align}\label{Surfacepartfouri}
\begin{split}
\sum_{\xi \in V} \underset{\mathrm{dist}(i,S(u))\leq  2R\varepsilon}{\sum_{ i \in Z_\varepsilon\left( \Omega \setminus \bigcup_{j=1}^N Q_j\right) }} \varepsilon^d W^\varepsilon_{i,\xi}(D^\xi_\varepsilon u_{\varepsilon}^{\delta,\eta}(i))&\leq C\varepsilon^{d-1}\# \left\{ \varepsilon\mathbb{Z}^d \cap (M)_{2R\varepsilon} \setminus \bigcup_{j=1}^N Q_j\right\} \\&\leq C \mathcal{H}^{d-1}(S(u) \setminus \bigcup_{j=1}^N Q_j) +C(u)\varepsilon \leq C\delta.
\end{split}
\end{align}
Using that $u \in SBV^2(\Omega)$ we have that
\begin{align*}
\sup_{\varepsilon >0} F_\varepsilon(u_{\varepsilon}^{\delta,\eta} ) <+\infty, \quad \sup_{\varepsilon>0} ||u_{\varepsilon}^{\delta,\eta}||_{L^\infty(\Omega)} < +\infty.
\end{align*}
Applying Lemma \ref{Compactnesslemma} we have that up to subsequences $u_{\varepsilon}^{\delta,\eta}$ converges to $u^{\delta,\eta}$ with respect to the strong $L^1(\Omega)$-topology as $\varepsilon \to 0$. And we have 
\begin{align}\label{udetabound}
\sup_{\eta>0}\left( \int_\Omega |\nabla u^{\delta,\eta}|^2\mathrm{d}x + \mathcal{H}^{d-1}(S(u^{\delta,\eta})) + ||u^{\delta,\eta}||_{L^\infty(\Omega)}\right) < +\infty.
\end{align}
Therefore by the Ambrosio Compactness Theorem (cf. \cite{ambrosio1989variational} Thm. 3.1) (up to subsequences) $u^{\delta,\eta}$ converges to $u^\delta$ with respect to the strong $L^1(\Omega)$-topology. We now want to prove that $u^\delta \to u$ with respect to the strong $L^1(\Omega)$-topology as $\delta \to 0$. We estimate, using the fact that $u^\delta = u $ on $\displaystyle \Omega \setminus \bigcup_{j=1}^N\Omega_j$ and Poincar\' e's Inequality,
\begin{align*}
||u^\delta-u||_{L^1(\Omega)} = \sum_{j=1}^N ||u^\delta-u||_{L^1(\Omega)} &\leq C\delta |Du^\delta-Du|(\bigcup_{j=1}^N\overline{Q_j}) \\&\leq C\delta\left(|Du^\delta|(\Omega) + |Du|(\Omega)\right).
\end{align*}
In view of (\ref{udetabound}) we have that $|Du^\delta|(\Omega)$ is bounded and we conclude that $u^\delta \to u$ with respect to the strong $L^1(\Omega)$ topology. Now
\begin{align*}
F^{\prime\prime}(u)&\leq \liminf_{\delta \to 0} F^{\prime\prime}(u^\delta) \leq \liminf_{\delta \to 0}\liminf_{\eta \to 0} F^{\prime\prime}(u^{\delta,\eta}) \leq \liminf_{\delta \to 0}\liminf_{\eta \to 0} \limsup_{\varepsilon \to 0} F_\varepsilon(u_\varepsilon ^{\delta,\eta}).
\end{align*}
Note that by (\ref{SurfaceLimsup})-(\ref{Surfacepartfouri}) and the dominated convergence theorem we have
\begin{align*}
\limsup_{\varepsilon \to 0}F_\varepsilon(u_\varepsilon ^{\delta,\eta}) \leq C&\left((1+ \mathcal{H}^{d-1}(S(u))+|\Omega|)\delta +\int_{\Omega \setminus \bigcup_{j=1}^{N} Q_j}|\nabla u(x)|^2 \mathrm{d}x \right) \\&+\int_{\Omega}f(x,\nabla u(x))\mathrm{d}x + \int_{S(u)}\varphi(x,\nu_u(x))\mathrm{d}\mathcal{H}^{d-1}\\&\leq F(u) + C\left((1+ \mathcal{H}^{d-1}(S(u))+|\Omega|)\delta +\int_{\Omega \setminus \bigcup_{j=1}^{N} Q_j}|\nabla u(x)|^2 \mathrm{d}x \right).
\end{align*}
Applying once more the dominated convergence theorem and (ii) we have that
\begin{align*}
|\Omega \setminus \bigcup_{j=1}^{N} Q_j| < \delta \text{ and therefore } \int_{\Omega \setminus \bigcup_{j=1}^{N} Q_j}|\nabla u(x)|^2 \mathrm{d}x \to 0
\end{align*}
as $\delta \to 0$. This yields the claim.
\end{proof}
Lemmas \ref{Neihbourhoodcardlemma} and \ref{UpperboundDxi} are needed in the proof of Proposition \ref{GammalimsupLemma}. We prove them in the following.

 \begin{proof}[Proof of Proposition \ref{Neihbourhoodcardlemma}] The proof follows from Lemma \ref{Cardinalitylemma} and noting that 
 \begin{align*}
 E_\rho = \bigcup_{x \in E} B_\rho(x),
\end{align*}   
so that conditions of Lemma \ref{Cardinalitylemma} are satisfied with $r=\rho$.
 \end{proof}
 \begin{lemma}\label{Cardinalitylemma}
 Let $\varepsilon>0$ and let $E \in \mathcal{B}(\mathbb{R}^d)$ be such that there exists $r>0$ such that for every $x \in E$ there exists a Ball $B_r=B_r^x \subset E$ such that $x \in B_r$. Then 
 \begin{align*}
 \#\left( E \cap \varepsilon\mathbb{Z}^d\right) \leq C\frac{|E|}{(\varepsilon\wedge r)^d},
 \end{align*}
 where $C$ only depends on $d$.
 \end{lemma}
\begin{proof}We have that 
\begin{align*}
|E| = \sum_{x \in \varepsilon\mathbb{Z}^d \cap E} |E \cap Q_{\varepsilon}(x)|\geq \sum_{x \in \varepsilon\mathbb{Z}^d \cap E} |B^x_r \cap Q_{\varepsilon}(x)| &\geq  \inf_{x \in B_r} |B_r \cap Q_{\varepsilon}(x)|\# \left(\varepsilon\mathbb{Z}^d \cap E\right) \\&=C_d (\varepsilon \wedge r)^d \# \left(\varepsilon\mathbb{Z}^d \cap E\right).
\end{align*}
Where $C_d = 2^{-2d} |B_1|$. It remains to prove
\begin{align*}
\inf_{x \in B_r} |B_r \cap Q_{\varepsilon}(x)| \geq 2^{-2d} |B_1|.
\end{align*}
Note that $Q_{\varepsilon} \supset B_{\frac{\varepsilon}{2}}$ and therefore, letting $\rho = r \wedge \frac{\varepsilon}{2}$, we have
\begin{align*}
\inf_{x \in B_r} |B_r \cap Q_{\varepsilon}(x)| \geq \inf_{x \in B_r} |B_r \cap B_{\frac{\varepsilon}{2}}(x)| \geq   \inf_{x \in B_\rho} |B_\rho \cap B_\rho(x)| \geq |B_{\frac{\rho}{2}}| \geq 2^{-2d}(\varepsilon \wedge r)^d |B_1|,
\end{align*}
where the second to last inequality follows since  we have $B_\rho \cap B_\rho(x) \supset B_{\frac{\rho}{2}}(\frac{x}{2})$ for $x \in B_\rho$. The last inequality follows by a scaling argument and the fact that $\frac{\varepsilon}{2} \wedge \rho \geq \frac{1}{2}(\varepsilon \wedge \rho)$. The claim follows by dividing by $C_d (\varepsilon \wedge r)^d$. 
\begin{figure}[htp]
\centering
\includegraphics{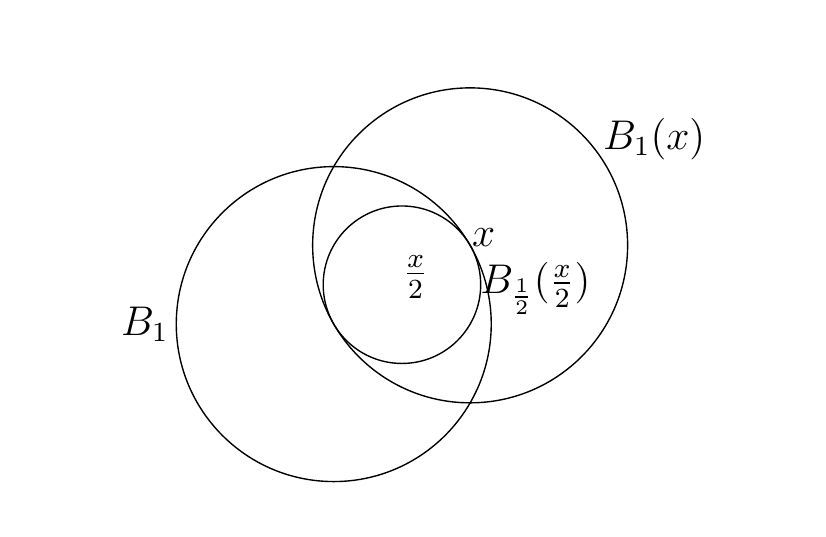}
\caption{Two intersecting balls where the centers are contained in the closure of both balls}
\end{figure}
\end{proof}

\begin{proof}[Proof of \ref{UpperboundDxi}] Let $\gamma : [0,1] \to \mathbb{R}^d $ be a lipschitz-continuous curve such that $\gamma(0)=i,\gamma(1)=i+\varepsilon\xi$, $\gamma([0,1]) \subset ([i,i+\varepsilon\xi])_\varepsilon$ and $\varepsilon\leq|\dot{\gamma}|\leq C|\xi|\varepsilon$. We have that $\mathrm{dist}([i,i+\varepsilon\xi],M) >\varepsilon$ and therefore, defining $v :[0,1] \to \mathbb{R}$ by
$
v(t) = u(\gamma(t))
$,
and using Jensen's Inequality, it holds by the fundamental theorem of calculus
\begin{align}\label{curveineq}
|D^\xi_\varepsilon u(i)|^2 &= \frac{1}{\varepsilon^2|\xi|^2}\left|\int_0^1\dot{v}(t)\mathrm{d}t \right|^2 = \frac{1}{\varepsilon^2|\xi|^2}\left|\int_0^1\nabla u(\gamma(t))\cdot \dot{\gamma}(t)\mathrm{d}t \right|^2  \\& \nonumber\leq \frac{1}{\varepsilon^2|\xi|^2}\int_0^1\left|\nabla u(\gamma(t))\right|^2 |\dot{\gamma}(t)|^2\mathrm{d}t \leq C\int_0^1 \left|\nabla u(\gamma(t))\right|^2 \mathrm{d}t\leq \frac{C}{\varepsilon} \int_\gamma |\nabla u(x)|^2\mathrm{d}\mathcal{H}^1.
\end{align}
The last inequality follows from the area formula and noting that $ \varepsilon\leq |\dot{\gamma}|$. Now, set $\nu_\xi = \frac{\xi}{||\xi||} \in S^{d-1}$ and let $ \nu \in ([i,i+\varepsilon\xi])_\varepsilon \cap \Pi^{\nu_\xi} (i+\frac{\varepsilon\xi}{2}) =H_\varepsilon^\xi(i) $. Note that $\nu_\xi \cdot \nu =0$ and $|\nu|\leq \varepsilon$. We define $\gamma_\nu : (0,1) \to \mathbb{R}^d$ by
\begin{align*}
\gamma_\nu(t) = \begin{cases} x+  t\varepsilon\xi+t\nu &t\in [0,\frac{1}{2})\\
x +t\varepsilon\xi + (1-t)\nu &t \in [\frac{1}{2},1].
\end{cases}
\end{align*}
Note that $\gamma_\nu$ is a lipschitz-continuous curve satisfying $\gamma_\nu(0)=i,\gamma_\nu(1)=i+\varepsilon\xi$, $\gamma_\nu([0,1]) \subset ([i,i+\varepsilon\xi])_\varepsilon$ and $\varepsilon\leq|\dot{\gamma_\nu}|\leq 2|\xi|\varepsilon$, $\gamma_{\nu_1}\cap \gamma_{\nu_2} = \{i,i+\varepsilon\xi\}$ for all $\nu_1 \neq \nu_2$ and therefore we can apply (\ref{curveineq}). Hence, noting that $C\mathcal{H}^{d-1}(H^\xi_\varepsilon(i))\geq \varepsilon^{d-1}$,  we get
\begin{align*}
\varepsilon^{d} |D^\xi_\varepsilon u(i)|^2 \leq \varepsilon C\int_{H_\varepsilon^\xi(i)}|D^\xi_\varepsilon u(i)|^2\mathrm{d}\mathcal{H}^{d-1}(\nu) &\leq C \int_{H_\varepsilon^\xi(i)} \int_{\gamma_\nu} |\nabla u(x)|^2\mathrm{d}\mathcal{H}^1\mathrm{d}\mathcal{H}^{d-1}(\nu) \\&\leq C\int_{([i,i+\varepsilon\xi])_\varepsilon} |\nabla u|^2 \mathrm{d}x
\end{align*}
and the claim follows.
\begin{figure}[htp]
\centering
\includegraphics{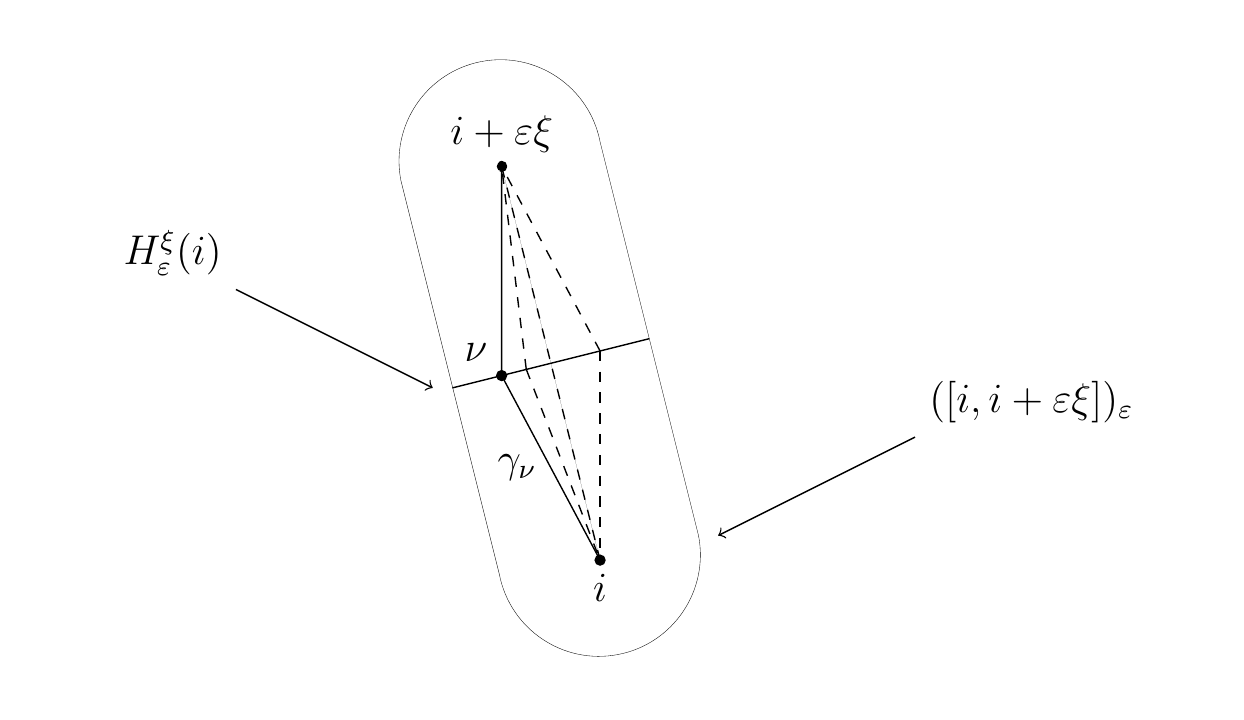}
\caption{The family of curves used to give the upper on $D^\xi_\varepsilon u(i)$ in Lemma \ref{UpperboundDxi}}
\end{figure}
\end{proof}

\begin{proof}[Proof of Lemma \ref{EqualityBulkLemma}] The proof follows essentially the same steps as the one of Lemma \ref{LiminfBulkLemma}. It suffices to prove the claim for $x_0 \in \Omega\setminus M$. Take such an $x_0$ and set $u_0(x) = u(x_0) + \nabla u(x_0)(x-x_0)$. Note that for $\rho>0$ small enough we have that both $u,u_0 \in C^\infty(\overline{Q^\nu_\rho(x_0)})$ and therefore $||\nabla u||_\infty,||\nabla u_0||_\infty \leq C$. Note that $u(x_0) =u_0(x_0)$ and therefore $||u-u_0||_\infty \leq C\rho$. By Lemma \ref{UpperboundDxi} we have that
\begin{align*}
m_\varepsilon^\eta(u,Q^\nu_\rho(x_0)) \leq C\rho^d, \quad m_\varepsilon^\eta(u_0,Q^\nu_\rho(x_0)) \leq C\rho^d.
\end{align*}
Now for every $\delta>0$ and a function $v \in \mathcal{PC}_\varepsilon(\Omega)$ such that $v_i = u_i $ on $(Q^\nu_\rho(x_0))_\eta$ and
\begin{align*}
F_\varepsilon(v,Q^\nu_\rho(x_0)) \leq m_\varepsilon^\eta(u,Q^\nu_\rho(x_0)) + o(\rho^d) \leq C\rho^d
\end{align*}
we can construct, using an analogous cut-off argument as in Lemma \ref{LiminfBulkLemma}, a function $w \in \mathcal{PC}_\varepsilon(\Omega)$ such that $w_i = (u_0)_i $ on $(Q^\nu_\rho(x_0))_\eta$ and
\begin{align*}
F_\varepsilon(w,Q^\nu_\rho(x_0)) \leq F_\varepsilon(v,Q^\nu_\rho(x_0)) + o(\rho^d) +\delta \rho^d.
\end{align*}
Note that in the proof of Lemma \ref{LiminfBulkLemma} it was important that the two functions we perform the cut-off construction with are close in $L^\infty$-norm, which is the case here, in order to let the error we commit by performing the cut-off go to $0$ as $\rho \to 0$.
From this follows 
\begin{align*}
\lim_{\rho \to 0} \frac{1}{\rho^d}\lim_{\eta \to 0} \lim_{\varepsilon \to 0} m_\varepsilon^\eta(u,Q^\nu_\rho(x_0)) \geq \lim_{\rho \to 0} \frac{1}{\rho^d}\lim_{\eta \to 0} \lim_{\varepsilon \to 0} m_\varepsilon^\eta(u(x_0)+\nabla u(x_0)(\cdot-x_0),Q^\nu_\rho(x_0)).
\end{align*}
Exchanging $u$ and $u_0$ and doing the same construction we obtain the other inequality.
\end{proof}

\begin{proof}[Proof of \ref{EqualitySurfaceLemma}] Fix a point in $S(u)$ and let $M$ be as in (\ref{D2def}). Let $\nu= \nu_u(x_0)$ and $\rho_0 > 0$ such that for all $\rho <\rho_0$ (\ref{D2def}) holds. 
Let $Q_{d-1,\rho}(x_0)$ the (d-1)-dimensional cube with side-length $\rho>0$ centred in $x_0$ and let  $f : Q_{d-1}(x_0) \to \mathbb{R}$ be a $C^1$-function such that after rotation $R \in SO(d)$ such that $R e_d=\nu$, setting $x=(x^{\prime},x_d)$, there holds
\begin{align*}
M \cap Q^\nu_\rho(x_0) &= R\{(x^{\prime},x_d) \in Q_\rho(x_0) : x_d = f(x^{\prime}) \},\\
Q^\pm_\rho(x_0) &= R\{(x^{\prime},x_d) \in Q_\rho(x_0) :  x_d \gtrless f(x^{\prime}) \}.
\end{align*}
Note that we can assume $f(x_0)=0$ and
\begin{align*}
R^{-1}\nu_M(x_0) = \frac{(\nabla f(x_0),1)}{\sqrt{1+|\nabla f(x_0)|^2}}=e_d,
\end{align*}
hence $\nabla f(x_0)=0$.
\begin{figure}[htp]
\centering
\includegraphics{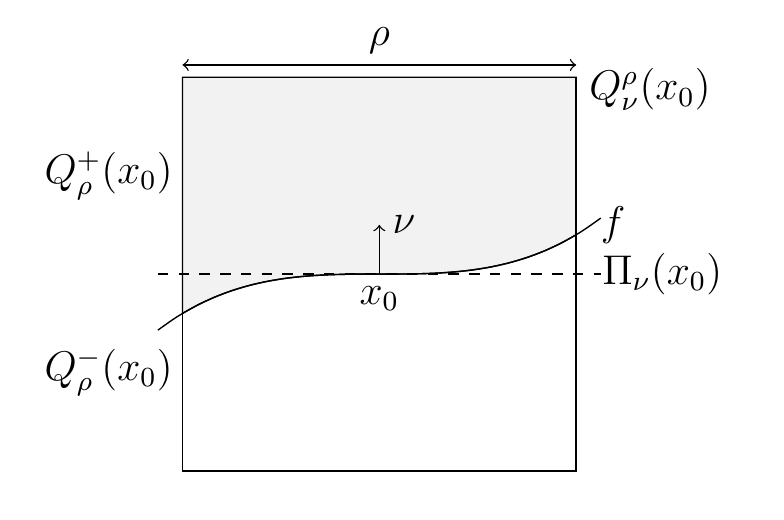}
\caption{Local representation of $S(u)$}
\end{figure}
Since $u \in \mathcal{D}_2(\Omega)$ we have that $u \in C^\infty(\overline{Q^\pm_\rho(x_0)})$ and therefore 
\begin{align}\label{uux0nucloseness}
\begin{split}
&|u(x)-u(y)| \leq C|x-y| \text{ for all } x,y \in Q^+_\rho(x_0), \\&|u(x)-u(y)| \leq C|x-y| \text{ for all } x,y \in Q^-_\rho(x_0).
\end{split}
\end{align}
Extend $u: Q_\rho^+ \to \mathbb{R}$ and $u: Q_\rho^- \to \mathbb{R}$  to Lipschitz functions $u_1 : \mathbb{R}^d \to \mathbb{R}$ and $u_2 : \mathbb{R}^d \to \mathbb{R}$ respectively. This can be done such that (\ref{uux0nucloseness}) holds for all $x,y \in \mathbb{R}^d$ for both $u_1$ and $u_2$. Define $w \in BV_{loc}(\mathbb{R}^d;\{-1,+1\})$ as an extension of
\begin{align*}
w(x) = 2 \chi_{Q_\rho^+(x_0)}(x) -1
\end{align*}
such that $S(w) \subset R\{f(x^{\prime})=x_d\}$ for all $x \in Q_{\rho_0}(x_0)$.
By the definition of $w$ and $Q^\pm_\rho(x_0)$ and Taylor expanding $f$ around $x_0$ we have that
\begin{align}\label{wtozero}
\begin{split}
\int_{Q^\nu_\rho(x_0)\setminus (Q^\nu_\rho(x_0))_\eta^+}|w-u_{x_0,\nu}| \mathrm{d}x &\leq 2( |(Q^\nu_\rho(x_0) \setminus (Q^\nu_\rho(x_0))_\eta^+)\cap  (Q^+_\rho(x_0) \cap \Pi_\nu^-(x_0))| \\&\quad+
|(Q^\nu_\rho(x_0) \setminus (Q^\nu_\rho(x_0))_\eta^+)\cap  (Q^-_\rho(x_0)\cap \Pi_\nu^+(x_0))|)\\& =2 \int_{Q_{d-1,\rho}(x_0)\setminus  (Q_{d-1,\rho}(x_0))^+_\eta} \int_{-|f(x^{\prime})|}^{|f(x^{\prime})|} \mathrm{d}x_d \mathrm{d}x^{\prime}
\\&= o\left(\int_{Q_{d-1,\rho}(x_0)\setminus  (Q_{d-1,\rho}(x_0))^+_\eta} |x^{\prime}-x_0^{\prime}| \mathrm{d}x^{\prime} \right) \\&\leq  o(\rho|Q_{d-1,\rho}(x_0)\setminus  (Q_{d-1,\rho}(x_0))^+_\eta|) = o(\eta\rho^{d-1}).
\end{split}
\end{align}
Let $u_\varepsilon^{\eta,\rho} \in \mathcal{PC}_\varepsilon(\mathbb{R}^d;\{-1,+1\})$ be such that $(u_\varepsilon^{\eta,\rho})_i = (u_{x_0,\nu})_i$ for all $i \in Z_\varepsilon(( Q^\nu_\rho(x_0))_\eta)$ and
\begin{align} \label{energybound uxnu}
F_\varepsilon(u_\varepsilon^{\eta,\rho},Q^\nu_\rho(x_0)) \leq m_\varepsilon^\eta(u_{x_0,\nu},Q^\nu_\rho(x_0)) + o(\rho^{d-1}) \leq C\rho^{d-1}
\end{align}
We construct $v_\varepsilon^{\eta,\rho} \in \mathcal{PC}_\varepsilon(\mathbb{R}^d)$ such that $(v_\varepsilon^{\eta,\rho})_i = u_i$ for all $i \in Z_\varepsilon( (Q^\nu_\rho(x_0))_{\frac{\eta}{2}})$ and
 \begin{align*}
 F_\varepsilon(v_\varepsilon^{\eta,\rho},Q^\nu_\rho(x_0)) \leq F_\varepsilon(u_\varepsilon^{\eta,\rho},Q^\nu_\rho(x_0)) +o(\rho^{d-1}) +C\eta \rho^{d-1}.
\end{align*}
Let $\displaystyle R=\sup_{\xi \in V} |\xi|$, $K_\varepsilon^\eta = \lfloor \frac{\eta}{6R\varepsilon} \rfloor \in \mathbb{N}$ and for $k \in \{ \frac{K_\varepsilon^\eta}{2} ,\ldots,K_\varepsilon^\eta\}$ we define $S_{k,\varepsilon} = (Q_\rho^\nu(x_0))^+_{3kR\varepsilon} \setminus (Q_\rho^\nu(x_0))^+_{3(k+1)R\varepsilon} $. For any $\varepsilon>0$ small enough and $\eta >0$ we have $w_\varepsilon \to w, (u_{x_0,\nu})_\varepsilon \to u_{x_0,\nu}$ in $L^1(Q^\nu_\rho(x_0))$ respectively and therefore
\begin{align}\label{integralcomparison}
\int_{Q^\nu_\rho(x_0)\setminus (Q^\nu_\rho(x_0))_\eta^+}|w-u_{x_0,\nu}| \mathrm{d}x \geq 2 \int_{Q^\nu_\rho(x_0)\setminus (Q^\nu_\rho(x_0))_\eta^+}|w_\varepsilon-(u_{x_0,\nu})_\varepsilon| \mathrm{d}x,
\end{align}
where $w_\varepsilon$ and $(u_{x_0,\nu})_\varepsilon$ are the discretizations of $w$ and $u_{x_0,\nu}$ respectively. Hence there exists $k(\varepsilon) \in \{ \frac{K_\varepsilon^\eta}{2} ,\ldots,K_\varepsilon^\eta\}$ such that
\begin{align}\label{cardestimatewuxnu}
\begin{split}
\int_{Q^\nu_\rho(x_0)\setminus (Q^\nu_\rho(x_0))_\eta^+}|w_\varepsilon-(u_{x_0,\nu})_\varepsilon| \mathrm{d}x &\geq  \sum_{k=\frac{K_\varepsilon^\eta}{2}}^{K_\varepsilon^\eta} \int_{S_{k,\varepsilon}}|w_\varepsilon-(u_{x_0,\nu})_\varepsilon| \mathrm{d}x \\&\geq  K_\varepsilon^\eta \sum_{i \in Z_\varepsilon( S_{k(\varepsilon),\varepsilon}) } \varepsilon^d |w_i - (u_{x_0,\nu})_i| \\&\geq  C \eta \varepsilon^{d-1} \#\{i \in Z_\varepsilon( S_{k(\varepsilon),\varepsilon}) : w_i \neq (u_{x_0,\nu})_i\}
\end{split}.
\end{align}
Define $ v^{\mathrm{aux}}_{\varepsilon,\eta,\rho} \in \mathcal{PC}_\varepsilon(\mathbb{R}^d;\{-1,+1\}) $ by 
\begin{align*}
v^{\mathrm{aux}}_{\varepsilon,\eta,\rho}(i) = \begin{cases} u^{\eta,\rho}_\varepsilon(i) &i \in Z_\varepsilon\left((Q_\rho^\nu(x_0))^+_{(3k(\varepsilon)+1)R\varepsilon}\right)\\
w(i) &\text{otherwise}.
\end{cases}
\end{align*}
We now have 
\begin{align}\label{Fepsauxestimate}
\begin{split}
F_\varepsilon(v^{\mathrm{aux}}_{\varepsilon,\eta,\rho},Q^\nu_\rho(x_0)) &\leq  F_\varepsilon(u^{\eta,\rho}_\varepsilon,Q^\nu_\rho(x_0)) + F_\varepsilon(w,(Q^\nu_\rho(x_0))_\eta)
\\& \quad + \sum_{\xi \in V} \underset{i+\varepsilon\xi \in S_{k(\varepsilon),\varepsilon}}{\sum_{i \in Z_\varepsilon\left( S_{k(\varepsilon),\varepsilon}\right)}} \varepsilon^d W^\varepsilon_{i,\xi}(D^\xi_\varepsilon v^{\mathrm{aux}}_{\varepsilon,\eta,\rho}(i)).
\end{split}
\end{align}
Note that it holds
\begin{align}\label{Wsubadd}
\begin{split}
W^\varepsilon_{i,\xi}(D^\xi_\varepsilon v^{\mathrm{aux}}_{\varepsilon,\eta,\rho}(i)) \leq   W^\varepsilon_{i,\xi}(D^\xi_\varepsilon w(i)) + \frac{c_{i,\xi}^\varepsilon}{\varepsilon}  ( \mathrm{1}_{w(i) \neq u^{\eta,\rho}_{\varepsilon}(i)} + \mathrm{1}_{w(i+\varepsilon\xi) \neq u^{\eta,\rho}_{\varepsilon}(i+\varepsilon\xi)})
\end{split}
\end{align}
and therefore, since $S_{k(\varepsilon),\varepsilon} \subset (Q^\nu_\rho(x_0))_\eta$, hence $u^{\eta,\rho}_\varepsilon(i) = (u_{x_0,\nu})(i)$ and using (\ref{wtozero}),(\ref{integralcomparison}) and (\ref{cardestimatewuxnu}), we have
\begin{align}\label{Stripeestimate}
 \sum_{\xi \in V} \underset{i+\varepsilon\xi \in Z_\varepsilon\left( S_{k(\varepsilon),\varepsilon}\right)}{\sum_{i \in Z_\varepsilon\left( S_{k(\varepsilon),\varepsilon}\right)}} \varepsilon^d W^\varepsilon_{i,\xi}(D^\xi_\varepsilon v^{\mathrm{aux}}_{\varepsilon,\eta,\rho}(i)) &\leq F_\varepsilon(w,(Q^\nu_\rho(x_0))_\eta)) \\&\nonumber+ C \varepsilon^{d-1} \#\{i \in Z_\varepsilon\left( S_{k(\varepsilon),\varepsilon}\right) : w_i \neq (u_{x_0,\nu})_i\}\\&\nonumber\leq F_\varepsilon(w,(Q^\nu_\rho(x_0))_\eta)) + o(\rho^{d-1})
\end{align}
Note that since $S(w)$ is the discretization of a $C^1$-manifold, using Lemma \ref{Neihbourhoodcardlemma}, we have that
\begin{align}\label{neighborhoodincube}
\begin{split}
F_\varepsilon(w,(Q^\nu_\rho(x_0))_\eta)) &\leq C\varepsilon^{d-1} \#\{\varepsilon\mathbb{Z}^s \cap (S(w)\cap (Q^\nu_\rho)_\eta)_{R\varepsilon} \} \leq \frac{C}{\varepsilon} | (S(w)\cap (Q^\nu_\rho)_\eta)_{R\varepsilon} | \\& \leq C\mathcal{H}^{d-1}(S(w) \cap (Q^\nu_\rho)_\eta)\leq C\eta\rho^{d-1}.
\end{split}
\end{align}
Using (\ref{Fepsauxestimate}), (\ref{Stripeestimate}) and (\ref{neighborhoodincube}) we obtain that
\begin{align}\label{Fauxcomparison}
F_\varepsilon(v^{\mathrm{aux}}_{\varepsilon,\eta,\rho},Q^\nu_\rho(x_0)) &\leq  F_\varepsilon(u^{\eta,\rho}_\varepsilon,Q^\nu_\rho(x_0)) +C \eta\rho^{d-1} + o(\rho^{d-1}).
\end{align}
We now define $v_\varepsilon^{\eta,\rho} \in \mathcal{PC}_\varepsilon(\mathbb{R}^d)$ by
\begin{align*}
v_\varepsilon^{\eta,\rho}(i) = \begin{cases} u_1(i) &v^{\mathrm{aux}}_{\varepsilon,\eta,\rho}(i)=+1 \\
u_2(i) &v^{\mathrm{aux}}_{\varepsilon,\eta,\rho}(i)=-1.
\end{cases}
\end{align*}
Note that $v^{\mathrm{aux}}_{\varepsilon,\eta,\rho}(i) = w(i)$ for $i \in (Q^\nu_\rho(x_0))_{\frac{\eta}{2}}$ and therefore $(v_\varepsilon^{\eta,\rho})_i = u_i$ for all $i \in Z_\varepsilon\left( (Q^\nu_\rho(x_0))_{\frac{\eta}{2}}\right)$. Assume moreover that $\varepsilon >0$ is small enough such that $ \displaystyle ||u_1-u_2||_\infty > \sup_{i\in \varepsilon\mathbb{Z}^d,\xi \in V} \sqrt{ c_{i,\xi}^\varepsilon \varepsilon}$. Note that for all $i \in Z_\varepsilon\left( Q^\nu_\rho(x_0)\right), \xi \in V$ there holds
\begin{align*}
W^\varepsilon_{i,\xi}( D^\xi_\varepsilon v_\varepsilon^{\eta,\rho}) \leq W^\varepsilon(D^\xi_\varepsilon v^{\mathrm{aux}}_{\varepsilon,\eta,\rho}(i)) + |D^\xi_\varepsilon u_1(i)|^2 +  |D^\xi_\varepsilon u_2(i)|^2.
\end{align*}
and therefore summing over $\xi \in V,i \in Z_\varepsilon\left( Q^\nu_\rho(x_0)\right)$ and using the Lipschitz continuity of $u_1,u_2$ to estimate
$
|D^\xi_\varepsilon u_1(i)|^2 +  |D^\xi_\varepsilon u_2(i)|^2 \leq C,
$
 we obtain
\begin{align}\label{Fvepsetarho}
F_\varepsilon(v_\varepsilon^{\eta,\rho},Q^\nu_\rho(x_0)) \leq F_\varepsilon(v^{\mathrm{aux}}_{\varepsilon,\eta,\rho},Q^\nu_\rho(x_0)) + C(u)\rho^d.
\end{align}
Dividing by $\rho^{d-1}$, using (\ref{energybound uxnu}), (\ref{Fepsauxestimate}) and (\ref{Fauxcomparison}), we obtain
\begin{align*}
\frac{1}{\rho^{d-1}} F_\varepsilon(v_\varepsilon^{\eta,\rho},Q^\nu_\rho(x_0)) &\leq \frac{1}{\rho^{d-1}} F_\varepsilon(u^{\eta,\rho}_\varepsilon,Q^\nu_\rho(x_0)) +C \eta + C(u)\rho + o(1) \\& \leq \frac{1}{\rho^{d-1}} m_\varepsilon^\eta(u_{x_0,\nu},Q_\rho^\nu(x_0)) +C \eta + C(u)\rho + o(1).
\end{align*}
Noting that 
$
F_\varepsilon(v_\varepsilon^{\eta,\rho},Q^\nu_\rho(x_0)) \geq m_\varepsilon^{\frac{\eta}{2}}(u,Q^\nu_\rho(x_0)),
$ the claim follows by letting first $\varepsilon \to 0$, $\eta \to 0$ and finally $\rho \to 0$.
\end{proof}

\section{Characterization of the Bulk Energy density}

This section is devoted to the characterization of the Bulk energy density. We want to prove that the elastic energy density can be recovered using only discrete functions, that do not have to high jumps.

\medskip

We start by introducing some notation and recalling a Theorem that is well known in the continuum setting.

\medskip

For $\varepsilon >0$ and $ v :\varepsilon\mathbb{Z}^d \to \mathbb{R} $ we define
\begin{align*}
\mathcal{M}_\varepsilon(u)(x) = \sup \left\{\frac{1}{\# Z_\varepsilon(Q_\eta(x)) }\sum_{z \in Z_\varepsilon(Q_\eta(x))} u(x) : \eta>0 \right\}.
\end{align*}
Furthermore for $v :\varepsilon\mathbb{Z}^d \to \mathbb{R}$ we define
\begin{align*}
|\nabla_\varepsilon v|(x) = \underset{|x-z| =\varepsilon}{\sum_{z \in \varepsilon\mathbb{Z}^d}}\frac{|u(x)-u(z)|}{|x-z|}.
\end{align*}

\begin{theorem}\label{ExtensionTheorem}
Let $\varepsilon >0$, $R>0$ and let $u : \varepsilon\mathbb{Z}^d \cap Q^\nu_{\rho_0} \to \mathbb{R} $ and let
\begin{align*}
E^\lambda_\varepsilon=\{x \in Q^\nu_{\rho_0} : \mathcal{M}_\varepsilon(|\nabla_\varepsilon u|)(x) \leq \lambda\}
\end{align*}
Then for any $\rho \in (0,\rho_0)$ we can find a Lipschitz function $v : \varepsilon\mathbb{Z}^d \cap Q^\nu_{\rho_0} \to \mathbb{R}$ such that $u(x)=v(x)$ for all $x \in E^\lambda_\varepsilon \cap Q^\nu_\rho$ and
\begin{align*}
\mathrm{lip}(v,Q) \leq c\lambda +C \frac{2||u||_\infty}{\rho_0-\rho}.
\end{align*}
\end{theorem}
\begin{proof}
Let $U : Z_\varepsilon(Q) \times Z_\varepsilon(Q) \to \mathbb{R}$ be defined by
\begin{align*}
U(x,z) = \frac{|u(x)-u(z)|}{|x-z|}.
\end{align*}
We need the inequality
\begin{align}\label{Maximumfunctionineq}
\frac{1}{\# Z_\varepsilon(Q_\eta(x)) }\sum_{z \in Z_\varepsilon(Q_\eta(x))} U(x,z) \leq C\mathcal{M}_\varepsilon(|\nabla_\varepsilon u|)(x)
\end{align}
which holds for any $0<\eta$ with some constant $C>0$ depending only on the dimension. This follows by repeating \cite{alicandro2004general}, Lemma 3.6 while noting that the geometry of $Q_\rho$ in order to perform the construction in Lemma 3.6 for all points in $Q_\rho$. We claim that for any $x,y \in Z_\varepsilon( E^\lambda_\varepsilon \cap Q_\rho)$ we have 
\begin{align}\label{Lipschitz Ineq}
|u(x)-u(y)| \leq \left( \frac{4\lambda}{\tilde{c}(d)} + \frac{2||u||_\infty}{\rho_0-\rho}\right)|x-y|.
\end{align}
Set $||x-y||_\infty=r$ and define 
\begin{align}\label{tildec}
\tilde{c}(d) = \#\left(Z_\varepsilon( Q_r(x) \cap Q_r(y))\right) \geq 1.
\end{align}
 assume $r < \rho_0-\rho$. We define
\begin{align*}
W_{x} : = \left\{z \in Z_\varepsilon( Q_r(x)) : U(x,z) > \frac{2\lambda}{C \tilde{c}(d)} \right\},W_{y} : = \left\{z \in Z_\varepsilon( Q_r(y)) : U(y,z) > \frac{2\lambda}{C \tilde{c}(d)} \right\}
\end{align*}
with $\tilde{c}(d)$ defined by (\ref{tildec}) and $C$ given by (\ref{Maximumfunctionineq}). By (\ref{Maximumfunctionineq}) we get
\begin{align*}
\# W_x < \frac{\tilde{c}(d)\mathcal{M}_\varepsilon(|\nabla_\varepsilon u|)(x)}{2\lambda}\leq \frac{\tilde{c}(d)}{2}.
\end{align*}
Similarly we get $\# W_y <\frac{\tilde{c}(d)}{2} $. Hence we can find $z \in Z_\varepsilon((Q_r(y) \cap Q_r(x) )\setminus (W_x\cup W_y))$. Since $|x-z| < r$, $|z-y|<r$ we get
\begin{align*}
U(x,y) < U(z,y) + U(z,x) < \frac{4\lambda}{C\tilde{c}(d)}
\end{align*}
and (\ref{Lipschitz Ineq}) follows. On the other hand if $r \geq \rho_0-\rho$, then
\begin{align*}
|u(x)-u(y)| \leq 2||u||_\infty \leq C\frac{2||u||_\infty}{\rho_0-\rho}|x-y|
\end{align*}
and we obtain (\ref{Lipschitz Ineq}). Now by the Kirszbraun Extension Theorem there exists a Lipschitz function with the required properties.
\end{proof}
Additionally to (H1)-(H3) assume there holds
\begin{itemize}
\item[(H4)] There exist $0<c<C$ such that $ c_{i,\xi}^\varepsilon \in [c,C] \cup\{0\}$ for all $i \in \varepsilon\mathbb{Z}^d, \xi \in V, \varepsilon>0$.
\end{itemize}
\begin{lemma} The function $f : \Omega \times \mathbb{R}^d \to [0,+\infty)$ and $h: \Omega \times \mathbb{R}^d \to [0,+\infty)$ given by (\ref{Definition f}) and (\ref{Definition h}) respectively are Carath\'eodory functions.
\end{lemma}
\begin{proof}
The fact that $h$ is Carath\'eodory function follows by \cite{alicandro2004general} Theorem 3.1, while for $f$ it follows following exactly the same steps as in \cite{ruf2017discrete} Lemma 3.8.
\end{proof}
The steps of the proof are essentially the same as the ones in \cite{ruf2017discrete}, Proposition 4. We state the proof here for completeness.
\begin{proposition}[Characterization of the Bulk-Energy density] \label{Prop Char BulkDensity} Assume (H1)-(H4) holds. Then for a.e. $x \in \Omega$ and all $\zeta \in \mathbb{R}^d$ there holds
\begin{align*}
h(x,\zeta) = f(x,\zeta).
\end{align*}
\end{proposition}
\begin{proof} Since $h$ and $f$ are both Carath\'eodory functions it suffices to prove the equality only for a dense set $\mathcal{D} \subset \mathbb{R}^d$. Take $\mathcal{D}= \mathbb{Q}^d$ and the set of points in $\Omega$ to be
\begin{align*}
\Omega\setminus N=\Omega\setminus \left(\bigcup_{\zeta \in \mathbb{Q}^d} N_
\zeta(\tilde{f}) \cup \bigcup_{\zeta \in \mathbb{Q}^d} N_\zeta(f)  \right),
\end{align*} 
where 
\begin{align*}
N_\zeta(h) = \left\{x \in \Omega : h(x,\zeta) \text{ is not a Lebesgue point for } h(\cdot,\zeta)\right\}.
\end{align*}
We then have that $|N|=0$ as $N$ being the countable union of nullsets. We first prove 
\begin{align*}
h(x,\zeta) \geq f(x,\zeta) \quad x \in \Omega, \zeta \in \mathbb{Q}^d.
\end{align*}
To this end let $u^\rho_\varepsilon \to \zeta\cdot$ in $L^1(\Omega)$ be such that 
\begin{align*}
\limsup_{\varepsilon \to 0} H_\varepsilon(u^\rho_\varepsilon,Q^\nu_\rho(x_0)) = H(\zeta\cdot,Q^\nu_\rho(x_0)).
\end{align*}
Since $F_\varepsilon(u,A) \leq H_\varepsilon(u,A)$ and $u_\varepsilon \to \zeta\cdot$ in $L^1(\Omega)$ by Proposition \ref{GammaliminfMS} we have that
\begin{align*}
F(\zeta\cdot,Q^\nu_\rho(x_0)) \leq \liminf_{\varepsilon \to 0} F_\varepsilon(u_\varepsilon,Q^\nu_\rho(x_0))\leq \limsup_{\varepsilon \to 0} H_\varepsilon(u_\varepsilon,Q^\nu_\rho(x_0)) = H(\zeta\cdot,Q^\nu_\rho(x_0)).
\end{align*}
Dividing by $\rho^d$ and letting $\rho \to 0$, while noting that
\begin{align*}
F(\zeta\cdot,Q^\nu_\rho(x_0))  = \int_{Q^\nu_\rho(x_0)} f(x,\zeta)\mathrm{d}x \leq \int_{Q^\nu_\rho(x_0)} h(x,\zeta)\mathrm{d}x
\end{align*}
we obtain the claim. Next we prove the opposite inequality. Set $ u_{x_0,\zeta}=\zeta \cdot (x-x_0)$ and let $\{\varepsilon_n\}_n \subset \{\varepsilon\}_\varepsilon$ and $u_\varepsilon \to  u_{x_0,\zeta}$ strongly in $L^1(\Omega)$ be such that
\begin{align*}
\limsup_{\varepsilon \to 0} F_\varepsilon(u_\varepsilon,Q^\nu_{\rho_0}(x_0)) = \lim_{\varepsilon \to 0} F_{\varepsilon_n}(u_{\varepsilon_n},Q^\nu_{\rho_0}(x_0))= F(u_{x_0,\zeta},Q^\nu_{\rho_0}(x_0)).
\end{align*}
Since truncation lowers the energy we can assume that $ ||u_\varepsilon||_\infty \leq C\rho $ and therefore we also have that $u_\varepsilon \to u_{x_0,\zeta}$ strongly in $L^2(Q^\nu_{\rho_0}(x_0))$. Now for $0<\rho <\rho_0$ using the same cut-off construction as in Lemma \ref{LiminfBulkLemma} we obtain
\begin{align*}
\limsup_{\varepsilon \to 0} F_\varepsilon(u_\varepsilon,Q^\nu_\rho(x_0)) &\leq \limsup_{\varepsilon\to 0} F_\varepsilon(u_\varepsilon,Q^\nu_{\rho_0}(x_0)) - \liminf_{\varepsilon \to 0} F_\varepsilon(u_\varepsilon, Q^\nu_{\rho_0} \setminus \overline{Q}^\nu_\rho(x_0))  \\&\leq F(u_{x_0,\zeta},Q^\nu_{\rho_0}(x_0)) - F(u_{x_0,\zeta},Q^\nu_{\rho_0}(x_0)\setminus \overline{Q}^\nu_{\rho}(x_0)) \\&= F(u_{x_0,\zeta},Q^\nu_\rho(x_0)).
\end{align*}
and therefore we obtain that $u_\varepsilon$ is a recovery sequence for $u_{x_0,\zeta}$ for all $0<\rho <\rho_0$. Choose $\rho_k \to 0$ such that
\begin{align*}
\lim_{k \to \infty} \frac{1}{\rho_k^d} F_{\rho_k}(u_{x_0,\zeta},Q^\nu_{\rho_k}(x_0)) = f(x_0,\zeta).
\end{align*}
and choose $\{\varepsilon_k\}_k \subset \{\varepsilon_n\}_n$ such that $\varepsilon_k \leq \rho_k$
\begin{align}\label{Energy and Convergencebound}
\begin{split}
&F_{\varepsilon_k}(u_{\varepsilon_k},Q^\nu_{\rho_k}) \leq C\rho_k^d, \quad ||u_{\varepsilon_k}-u_{x_0,\zeta}||_{L^2(Q^\nu_{\rho_k}(x_0))}^2 \leq \rho_k^{d+3} \\
\lim_{k \to \infty} \frac{1}{\rho_k^d} &F_{\varepsilon_k}(u_{\varepsilon_k},Q^\nu_{\rho_k}(x_0)) \leq \lim_{k\to \infty} \frac{1}{\rho_k^d} F_{\varepsilon_k}(u_{x_0,\zeta},Q^\nu_{\rho_k}(x_0)) = f(x_0,\zeta)
\end{split}
\end{align}
and
\begin{align*}
\lim_{k \to \infty}\frac{1}{\rho^d_k} m^{H_{\varepsilon_k}}_{\varepsilon_k,\varepsilon_k}(\zeta\cdot,Q^\nu_{\rho_k}(x_0)) = \lim_{\rho \to 0}\frac{1}{\rho^d}\limsup_{\varepsilon \to 0} m^{H_\varepsilon}_{\varepsilon,\varepsilon}(\zeta\cdot,Q^\nu_\rho(x_0)) = h(x,\zeta).
\end{align*}
Next we construct discrete Lipschitz competitors that still have less energy than the recovery sequence $u_{\varepsilon_k}$. 

\medskip

\textbf{Notation:}
For a function $u : Z_\varepsilon(A) \to \mathbb{R}$ we write
\begin{align*}
u\left(A\right) = \sum_{i \in Z_\varepsilon(A)}u(i)
\end{align*}

\medskip

The strategy of the proof is to use a discrete Lusin approximation of $BV$-functions in order to construct a sequence $v_\varepsilon$ such that $v_\varepsilon = \zeta(\cdot-x_0) $ on $(\partial Q^\nu_\rho(x_0))_\eta$  and
\begin{align*}
\limsup_{\varepsilon \to 0} H_\varepsilon(v_\varepsilon,Q^\nu_\rho(x_0)) \leq \limsup_{\varepsilon \to 0} F_\varepsilon(u_\varepsilon,Q^\nu_\rho(x_0)) + o(\rho^{d}).
\end{align*}

\textbf{Step 1:} Construction of a Lipschitz Competitor. Fix $\lambda >0$ and define 
\begin{align*}
&R_{k}^\lambda = \left\{i \in Z_{\varepsilon_k}(Q^\nu_{\rho_k}(x_0)) : \mathcal{M}_{\varepsilon_k}|
\nabla_{\varepsilon_k} u_{\varepsilon_k}|(i) > \lambda \right\},\\&
S_k^\lambda = \left\{i \in Z_{\varepsilon_k}(Q^\nu_{\rho_k}(x_0)) : |\nabla_{\varepsilon_k} u_{\varepsilon_k}|(i) \leq \frac{\lambda}{2} \right\}.
\end{align*}
Arguing as in the continuum we can estimate the cardinality of $R_k^\lambda$ with the (discrete) $L^1$-norm of the gradient. For every $i \in Z_{\varepsilon_k}(Q^\nu_{\rho_k}(x_0)) \setminus R_k^\lambda$ there exists $0 < \eta_i$ such that $Q_{\eta_i}(i) \subset Q^\nu_{\rho_k}(x_0)$ and
\begin{align}\label{Cardboundlambda}
\lambda \# Z_{\varepsilon_k}( Q_{\eta_i}(i)) < |\nabla_\varepsilon u_{\varepsilon_k}|\left(Q_{\eta_i}(i)\right)
\end{align}
By Vitalis Covering Theorem there exists a finite collection of disjoint cubes $Q_{\eta_i}(i)$ with $\{i_n^k\}_{n=1}^{N_k} \in Z_{\varepsilon_k}(Q^\nu_{\rho_k}(x_0)) \setminus R_k^\lambda$ satisfying (\ref{Cardboundlambda}) and
\begin{align}\label{Rlambdasubsetcubes}
Z_{\varepsilon_k}(Q^\nu_{\rho_k}(x_0)) \setminus R_k^\lambda \subset \bigcup_{n=1}^{N_k} Q_{5\eta_{i_n}}(i_n).
\end{align}
Since the cubes are disjoint, using the definition of $S^\lambda_k$, we conclude that
\begin{align*}
\lambda \#\left(\bigcup_{n=1}^{N_k} Q_{\eta_{i_n}}(i_n)\right) &< |\nabla_{\varepsilon_k} u_{\varepsilon_k}|\left(\bigcup_n Q_{\eta_{i_n}}(i_n) \cap S^\lambda_k\right)\\&\leq |\nabla_{\varepsilon_k} u_{\varepsilon_k}|\left(\bigcup_nQ_{\eta_{i_n}}(i_n) \cap S^\lambda_k\right) + \frac{\lambda}{2} \#\left(\bigcup_{n} Q_{\eta_{i_n}}(i_n)\right).
\end{align*}
Rearranging the terms we obtain
\begin{align}\label{Cardinalitycubes}
\#Z_{\varepsilon_k}\left(\bigcup_{n=1}^{N_k} Q_{\eta_{i_n}}(i_n)\right) \leq \frac{2}{\lambda} |\nabla_{\varepsilon_k} u_{\varepsilon_k}|\left(\bigcup_nQ_{\eta_{i_n}}(i_n) \cap S^\lambda_k\right)
\end{align}
Define
\begin{align*}
\mathcal{B}_k= \left\{i \in Z_{\varepsilon_k}(Q_{\rho_k}(x_0)) : |\nabla_{\varepsilon_k} u_{\varepsilon_k}|^2(i) \geq \varepsilon_k^{-1} \right\}
\end{align*}
For $i \in Z_{\varepsilon_k}(Q_{\rho_k}(x_0))$ we have that there exists $\xi \in V$ such that $|D^\xi_{\varepsilon_k} u_{\varepsilon_k}(i)|^2 \geq c \varepsilon_k^{-1}$ and therefore
$
W^{\varepsilon_k}_{i,\xi}(|D^\xi_{\varepsilon_k} u_{\varepsilon_k}(i)|) \geq c\varepsilon_k^{-1}.
$
Hence we obtain that
\begin{align*}
\varepsilon_k^{d-1} \#\left(\mathcal{B}_k \right) \leq C F_{\varepsilon_k}(u_{\varepsilon_k},Q^\nu_{\rho_k}(x_0)) \leq C\rho_k^d.
\end{align*}
Since $||u_{\varepsilon_k}||_\infty \leq C\rho_k$ we have that $|D_{\varepsilon_k}^\xi u_{\varepsilon_k}(i)| \leq C\varepsilon_k^{-1} \rho_k$ and therefore by (\ref{Energy and Convergencebound}) we have
\begin{align}\label{Bkbound}
|\nabla_{\varepsilon_k} u_{\varepsilon_k}|\left(\mathcal{B}_k\right) \leq C \left(\frac{\rho_k}{\varepsilon_k}\right)^d \rho_k.
\end{align}
On the other hand by H\"older's Inequality we obtain 
\begin{align}\label{IntGradbound}
|\nabla_{\varepsilon_k}u_{\varepsilon_k}|\left(S^\lambda_k\setminus \mathcal{B}_k\right)  \leq\# \left(S^\lambda_k \setminus \mathcal{B}_k\right)^{\frac{1}{2}} \left(|\nabla_{\varepsilon_k}u_{\varepsilon_k}|^2\left(Q^\nu_{\rho_k}(x_0) \setminus \mathcal{B}_k\right)\right)^{\frac{1}{2}}
\end{align}
Now for $i \notin \mathcal{B}_k$ we obtain that $|D^\xi_{\varepsilon_k} u(i)|^2 \leq \varepsilon_k^{-1}$ for all $\xi \in V$. Hence we infer from the definition of $W_{i,\xi}^\varepsilon$ that for $i \notin \mathcal{B}_k$ there holds
\begin{align*}
|\nabla_{\varepsilon_k} u_{\varepsilon_k}|^2(i) \leq C\sum_{\xi \in V} W_{i,\xi}^\varepsilon(D^\xi_\varepsilon u_{\varepsilon_k}(i)).
\end{align*}
Thus again with (\ref{Energy and Convergencebound}) we obtain
\begin{align}\label{Intsquarebound}
 \left(|\nabla_{\varepsilon_k}u_{\varepsilon_k}|^2\left(Q^\nu_{\rho_k}(x_0) \setminus \mathcal{B}_k \right)\right)^{\frac{1}{2}} \leq C \varepsilon_k^{-\frac{d}{2}} F_{\varepsilon_k}(u_{\varepsilon_k},Q^\nu_{\rho_k}(x_0))^{\frac{1}{2}}\leq C \left(\frac{\rho_k}{\varepsilon_k}\right)^{\frac{d}{2}}.
\end{align}
Combining this with (\ref{IntGradbound}) we obtain
\begin{align}\label{Estimate Gradient}
|\nabla_{\varepsilon_k}u_{\varepsilon_k}|\left(S^\lambda_k \setminus \mathcal{B}_k \right) \leq C \# \left(S^\lambda_k \setminus \mathcal{B}_k\right)^{\frac{1}{2}} \left(\frac{\rho_k}{\varepsilon_k}\right)^{\frac{d}{2}}.
\end{align}
Using (\ref{Intsquarebound}) and the definition of $S^\lambda_k$ we obtain
\begin{align*}
\#\left(S^\lambda_k\setminus \mathcal{B}_k\right)\frac{\lambda^2}{4} \leq |\nabla_{\varepsilon_k}u_{\varepsilon_k}|^2\left(Q^\nu_{\rho_k}(x_0) \setminus \mathcal{B}_k \right) \leq C\left(\frac{\rho_k}{\varepsilon_k}\right)^{d}
\end{align*}
Plugging this into (\ref{Estimate Gradient}) we obtain
\begin{align}\label{finalbound}
|\nabla_{\varepsilon_k}u_{\varepsilon_k}|\left(S^\lambda_k \setminus \mathcal{B}_k \right) \leq C \lambda^{-1}\left(\frac{\rho_k}{\varepsilon_k}\right)^{d}.
\end{align}
Using (\ref{Rlambdasubsetcubes}), (\ref{Cardinalitycubes}), splitting $S^\lambda_k$ into $S^\lambda_k \cap \mathcal{B}_k$ and $S^\lambda_k \setminus \mathcal{B}_k$ and estimating the cardinality separately using (\ref{Bkbound}) and (\ref{finalbound}) we obtain
\begin{align*}
\#\left(Z_{\varepsilon_k}(Q^\nu_{\rho_k}(x_0)) \setminus R_k^\lambda\right)\leq \#Z_{\varepsilon_k}\left(\bigcup_{n=1}^{N_k} Q_{5\eta_{i_n}}(i_n)\right) &\leq C\#Z_{\varepsilon_k}\left(\bigcup_{n=1}^{N_k} Q_{\eta_{i_n}}(i_n)\right) \\&\leq C\left(\frac{\rho_k}{\varepsilon_k}\right)^d\left(\rho_k\lambda^{-1}+\lambda^{-2}\right).
\end{align*}
Choosing $\lambda=\lambda_k= \rho_k^{-1}$ we obtain
\begin{align}\label{CardQsenzaRlambda}
\#\left(Z_{\varepsilon_k}(Q^\nu_{\rho_k}(x_0)) \setminus R_k^\lambda\right)\leq C\left(\frac{\rho_k}{\varepsilon_k}\right)^d \lambda_k^{-2} = C\left(\frac{\rho_k}{\varepsilon_k}\right)^d \rho_k^2  .
\end{align}
Using now Theorem \ref{ExtensionTheorem} we obtain a function $v_{\varepsilon_k} : Z_{\varepsilon_k}(\mathbb{R}^d) \to \mathbb{R} $ such that $\mathrm{lip (v_{\varepsilon_k}})\leq C\rho_k^{-1}$ and $v_{\varepsilon_k} = u_{\varepsilon_k}$ on $R^{\lambda_k}_k$. Moreover again truncating if necessary we can assume that $||v_{\varepsilon_k}||_\infty \leq C\rho_k$.

\medskip

\textbf{Step 2:} Construction of a competitor whose discrete gradients are equi-integrable in $L^2$.

\medskip

In order to modify the functions $v_{\varepsilon_k}$ constructed in the first step we rescale them. To this end we set $\eta_k = \frac{\varepsilon_k}{\rho_k}$ and $Z_{k}' = \eta_k\mathbb{Z}^d - \rho_k^{-1}x_0$ and define $w_{k} : Z_{k} \to \mathbb{R}$ by
\begin{align*}
w_{k}(x) = \rho_k^{-1} v_{\varepsilon_k}(\rho_k x +x_0).
\end{align*}
Note that by the properties of $v_{\varepsilon_k}$ we have that 
\begin{itemize}
\item[i)] $||w_k||_\infty\leq C$
\item[ii)] $|w_k(x) - w_k(y)| \leq C\rho_k^{-1}$ for all $x,y \in Z_k'$
\item[iii)]
$\displaystyle
\rho_k^{-d}\sum_{i \in Z_{\varepsilon_k}(Q^
\nu_{\rho_k}(x_0))} \varepsilon^d_{k} |\nabla_{\varepsilon_k} v_{\varepsilon_k}|^2 = \sum_{i \in Z_k' \cap Q^\nu_1}\eta_k^d  |\nabla_{\eta_k} w_k|^2
$
\item[iv)] $\rho_kw_k(\rho_k^{-1}(x-x_0)) = v_{\varepsilon_k}(x)=  u_{\varepsilon_k}(x)$ for all $x \in R^{\lambda_k}_k $.
\end{itemize}
By ii) we have that $|| |\nabla_{\eta_k} w| ||_\infty \leq C\rho_k^{-1}$.
Now extending $|\nabla_{\eta_k} w|$ piecewise constantly on the cubes $Q_{\eta_k}(x), x \in Z_k'$, viewing it as an element of $L^2(\mathbb{R}^d)$ using iii) and iv) we obtain
\begin{align*}
||\, |\nabla_{\eta_k} w_k|\, ||_{L^2(Q_1^\nu)}^2 &= \sum_{i \in Z_k' \cap Q^\nu_1}\eta_k^d  |\nabla_{\eta_k} w_k|^2 = \sum_{i \in Z_{\varepsilon_k}(Q^
\nu_{\rho_k}(x_0))} \varepsilon^d_{k} |\nabla_{\varepsilon_k} v_{\varepsilon_k}|^2 \\& \leq C\rho_k^{-2}\#\left(Z_{\varepsilon_k}(Q^\nu_{\rho_k}(x_0)) \setminus R_k^\lambda\right) + \sum_{i \in R_k^{\lambda_k}}\varepsilon^d_{k} |\nabla_{\varepsilon_k} v_{\varepsilon_k}|^2.
\end{align*}
Since we have that $|\nabla_{\varepsilon_k} v_{\varepsilon_k}|(i)=|\nabla_{\varepsilon_k} u_{\varepsilon_k}|(i) \leq \mathcal{M}_{\varepsilon_k} |\nabla_{\varepsilon_k} u_{\varepsilon_k}|(i) \leq \rho_k^{-1}\leq \varepsilon_k^{-1}$ for $i \in R^{\lambda_k}_k$ and hence $R^{\lambda_k}_k \subset Z_{\varepsilon_k}(Q^\nu_{\rho_k}(x_0))\setminus \mathcal{B}_k$. Therefore we can use (\ref{Intsquarebound}) and (\ref{CardQsenzaRlambda}) to obtain 
\begin{align*}
||\, |\nabla_{\eta_k} w_k|\, ||_{L^2(Q_1^\nu)}^2 \leq C.
\end{align*}
Using \cite{stein2016harmonic},Theorem 3.1 we have that
\begin{align*}
||\mathcal{M}_{\eta_k} |\nabla_{\eta_k} w_k|\, ||_{L^2(Q_1^\nu)}^2 \leq C||\, |\nabla_{\eta_k} w_k|\, ||_{L^2(Q_1^\nu)}^2 \leq C.
\end{align*}
Applying \cite{fonseca2007modern} Lemma 2.31 there exists a subsequence $\{k\}_k$ not relabelled and an increasing sequence of positive integers $l_k \to \infty$ such that the sequence $\left(\left(\mathcal{M}_{\eta_k} |\nabla_{\eta_k} w_k|\right)_{l_k}\right)^2$ is equi-integrable on $Q^\nu_1$. We need to modify the sequence $w_k$. To this end define
\begin{align*}
R_k = \left\{i \in Z'_k\cap Q^\nu_1 : \mathcal{M}_{\eta_k}|\nabla_{\eta_k}w_k|(i)\leq l_k\right\}.
\end{align*}
Viewing $\mathcal{M}_{\eta_k}|\nabla_{\eta_k}w_k|$ as an element of $L^2(Q^\nu_1)$ we have
\begin{align}\label{CardRk}
\eta_k^d\#\left(Z' \cap Q^\nu_1 \setminus R_k\right) \leq \frac{1}{l_k^2}\int_{Q^\nu_1}\mathcal{M}_{\eta_k}|\nabla_{\eta_k}w_k|^2(x)\mathrm{d}x \leq \frac{C}{l_k^2}.
\end{align}
Note that if $i \in R_k + Q_{R\eta_k}$ we have either $\mathcal{M}_{\eta_k}|\nabla_{\eta_k}w_k|(i)\leq l_k$ or there exists $i' =i+\eta_k \xi$ with $||\xi||_\infty \leq R$ such that $\mathcal{M}_{\eta_k}|\nabla_{\eta_k}w_k|(i') \leq l_k$. Noting that for $r > R\eta_k$ we have
$
Q_{r}(i) \subset Q_{2r}(i') 
$
and therefore 
\begin{align*}
\sup_{r >R\eta_k} \frac{1}{\#(Z'_k \cap Q_r(i)) } \sum_{j \in Z'_k \cap Q_r(i)} |\nabla_{\eta_k}w_k|(j) &\leq  \sup_{r >R\eta_k} \frac{C}{\#(Z'_k \cap Q_{2r}(i')) } \sum_{j \in Z'_k \cap Q_{2r}(i')} |\nabla_{\eta_k}w_k|(j) \\&\leq C\mathcal{M}_{\eta_k}|\nabla_{\eta_k}w_k|(i') \leq C l_k,
\end{align*}
where we used that $\# \left(Z'_k \cap Q_{2r}(i')\right) \leq C  \# \left(Z'_k \cap Q_{r}(i')\right)$.
If $r\leq R\eta_k$ we have
\begin{align*}
\mathcal{M}_{\eta_k}|\nabla_{\eta_k}w_k|(i) &= \frac{1}{\#\left(Z_k'\cap  Q_r(i)\right)}\sum_{j \in Z_k'\cap  Q_r(i)} |\nabla_{\eta_k}w_k|(j) \\&\leq \frac{C}{\# \left(Z'_k \cap Q_{4R\eta_k}(i)\right)} \sum_{j \in Z'_k \cap Q_{4R\eta_k}(i)} |\nabla_{\eta_k}w_k|(j) \\&\leq \frac{C}{\# \left(Z'_k \cap Q_{8R\eta_k}(i')\right)}\sum_{j \in Z'_k \cap Q_{8R\eta_k}(i')} |\nabla_{\eta_k}w_k|(j),
\end{align*}
where we used that there exists $C=C(R)>0$ such that $C^{-1}(R)\#\left(Z_k'\cap Q_{8C\eta_k}(i')\right)\leq  \#\left(Z_k'\cap Q_{4C\eta_k}(i)\right)\leq C(R)$ so that for $y \in R_k +Q_{R\eta_k}$ we have
\begin{align*}
M_{\eta_k} |\nabla_{\eta_k}w_k|(i) \leq C(R)l_k.
\end{align*}
Now again by Theorem \ref{ExtensionTheorem} and the Kirszbraun's extension Theorem we find a sequence Lipschitz functions $u_k : Z'_k \cap Q^\nu_1 \to \mathbb{R}$ such that $u_k(\cdot) =  \rho_k^{-1} u_{\varepsilon_k}(\rho_k\cdot +x_0)$ on $R_k +Q_{R\eta_k}$ and $\mathrm{lip}(u_k)\leq C(R)l_k$. Moreover we can assume that $||u_k||_\infty \leq C$. Note that for $i \in R_k$ we have
\begin{align*}
|\nabla_{\eta_k}u_k|(i) \leq M_{\eta_k} |\nabla_{\eta_k}w_k|(i) = \left(M_{\eta_k} |\nabla_{\eta_k}w_k|(i)\right)_{l_k},
\end{align*}
while for $i \in \left(Z_k' Q^\nu_1\right)\setminus R_k$ there holds
\begin{align*}
|\nabla_{\eta_k}u_k|(i) \leq Cl_k = C \left(M_{\eta_k} |\nabla_{\eta_k}w_k|(i)\right)_{l_k}.
\end{align*}
We therefore have that also the sequence $|\nabla_{\eta_k}u_k|^2$ is equi-integrable.

\textbf{Step 3:} Energy inequality

\medskip
Fix $R> \max\{||\xi||_\infty : \xi \in V\}$ and write $S_k= (R_k +Q_{R\eta_k}) \cap(\rho_k R^{\lambda_k}_k-x_0)$.
First we check that $w_k$ converges to $M\cdot$ in $L^2(Q^\nu_1)$. Using the $L^\infty$ bound of $u_k$, splitting the set into 
\begin{align*}
Q^\nu_1 \cap Z_k' \subset\left( Q^\nu_1 \cap S_k\right) \cup \left( Z_k'\cap Q^\nu_1 \setminus R_k\right) \cup \left( Z_k'\cap Q^\nu_1 \setminus (\rho_k R^{\lambda_k}_k-x_0)\right)   
\end{align*}
using (\ref{Energy and Convergencebound}),(\ref{CardQsenzaRlambda}) and (\ref{CardRk}) we obtain
\begin{align*}
||u_k - u_{x_0,\zeta}||_{L^2(Q^\nu_1)}^2 &\leq C\rho^{-d-2}_k\int_{Q^\nu_\rho(x_0)}|u_{\varepsilon_k}(x)-u_{x_0,\zeta}(x)|^2\mathrm{d}x \\&\quad+ C\eta_k^d\#(Z_k'\cap Q^\nu_1 \setminus R_k)  + \eta_k^d \#(Z_{\varepsilon_k}(Q^\nu_{\rho_k}(x_0))\setminus R^{\lambda_k}_k) \\&\leq C\rho_k + \frac{C}{l_k^2} + C\rho_k^2
\end{align*}
and therefore $u_k \to u_{x_0,\zeta}$ in $L^1(Q^\nu_1)$. Fix $\delta>0$ we have that 
\begin{align*}
\left|\bigcup_{j \in Z_k' \cap Q^\nu_1\setminus S_k}Q_{\eta_k}(j)\right|=\lim_k\eta_k^d \#(Z_k' \cap Q^\nu_1 \setminus S_k) \leq \lim_k C(l_k^{-2} +\rho^2_k) = 0
\end{align*}
using the equi-integrability of $|\nabla_{\eta_k} u_k|^2$ we have that there exists $k_\delta \in \mathbb{N}$ such that for all $k \geq k_\delta$ there holds 
\begin{align} \label{Bigzero}
\sum_{j \in Z_k'\cap Q^\nu_1\setminus S_k}\eta_k^d |\nabla_{\eta_k}u_k|^2 \leq C\int_{\bigcup_{j \in Z_k' \cap Q^\nu_1 S_k}Q_{\eta_k}(j)}|\nabla_{\eta_k}u_k|^2\mathrm{d}x \leq \delta.
\end{align}
For $t>0$ we set
\begin{align*}
A_k(t) = \left\{i \in Z_k'\cap Q^\nu_1  : |\nabla_{\eta_k} u_k|^2(i) >t \right\} . 
\end{align*}
Due to the equiintegrability established in Step 2 we have that there exists $t_\delta > 0$ such that for all $k \geq k_\delta$ there holds
\begin{align}\label{tdeltabound}
\sum_{j \in Z_k'\cap Q^\nu_1\cap \left(A_k(t_\delta)+Q_{R\eta_k}\right)}\eta_k^d |\nabla_{\eta_k}u_k|^2 \leq C\int_{Q^\nu_1 \cap \{|\nabla_{\eta_k}u_k|^2 >t_\delta\}}|\nabla_{\eta_k}u_k|^2\mathrm{d}x \leq \delta.
\end{align}
If $i \in Q^\nu_1 \cap S_k \setminus \left(A_k(t_\delta) +Q_{R\eta_k}\right)$ there holds
\begin{align*}
|D^\xi_{\varepsilon_k} u_{\varepsilon_k}|^2(\rho_k i +x_0) =|D^\xi_{\eta_k} u_k|^2(i) \leq C t_\delta < c\varepsilon_k^{-1}
\end{align*}
for $k $ large enough. The first inequality follows by choosing a path $(i_h)_{h=1}^{||\xi||_1}$ such that $i_1=i, i_N = i+\varepsilon_k\xi$, $i_{n+1}=i_n +e_{j(n)}$ and noting that by Jensen's Inequality there holds
\begin{align*}
|D^\xi_\varepsilon u_{\varepsilon_k}(i)|^2 &=  \frac{1}{|\xi|^2}\left|\sum_{n=1}^{||\xi||_1-1} D^{e_{j(n)}}_{\varepsilon_k}u_{\varepsilon_k}(i_n) \right|^2 \leq  \frac{||\xi_1||_1^2}{|\xi|^2}\sum_{n=1}^{||\xi||_1-1}\left| D^{e_{j(n)}}_{\varepsilon_k}u_{\varepsilon_k}(i_n) \right|^2 \\&\leq C(d) \sum_{n=1}^{||\xi||_1-1}\left| \nabla_{\varepsilon_k} u_{\varepsilon_k}(i_n) \right|^2.
\end{align*}
We therefore have that
\begin{align}\label{goodsetequality}
W_{\rho_k i+x_0,\xi}^\varepsilon(D^\xi_{\varepsilon_k} u_{\varepsilon_k}(\rho_k i +x_0)) = |D^\xi_{\eta_k} u_k|^2(i)
\end{align}
for all $k$ large enough. Performing a change of variables to define
\begin{align*}
w_{\varepsilon_k}(i) = u_k(\rho_k^{-1}(i-x_0)), \quad i \in Z_{\varepsilon_k}(Q^\nu_{\rho_k}(x_0))
\end{align*}
and performing a cut-off construction as in Lemma \ref{LiminfBulkLemma} we have that for $\eta >0$ we have that
\begin{align*}
m^{H_{\varepsilon_k}}_{\varepsilon_k,\varepsilon_k}(u_{x_0,\zeta},Q^\nu_{\rho_k}(x_0)) \leq H_{\varepsilon_k}(w_{\varepsilon_k},Q^\nu_{\rho_k}(x_0)) + o(\rho_k^d).
\end{align*}
Now using (\ref{Energy and Convergencebound}), (\ref{Bigzero})-(\ref{goodsetequality}) we obtain for $k$ large enough
\begin{align*}
\rho^{-d}_k F_{\varepsilon_k}(u_{\varepsilon_k},Q^\nu_{\rho_k}(x_0)) \geq \rho_k^{-d} H_{\varepsilon_k}(w_{\varepsilon_k},Q^\nu_{\rho_k}(x_0)) -3\delta \geq \rho_k^{-d} m^{H_{\varepsilon_k}}_{\varepsilon_k,\varepsilon_k}(u_{x_0,\zeta},Q^\nu_{\rho_k}(x_0)) - 3\delta.
\end{align*}
The claim follows by taking $k \to \infty$ and $\delta \to 0$.
\end{proof}

\textbf{Acknowledgements:} The author thanks Roberto Alicandro, Andrea Braides as well as Manuel Friedrich  for many fruitful discussions. The author acknowledges support from the Austrian Science Fund (FWF) project P 29681 and the fact that this work has been funded by the Vienna Science and Technology Fund (WWTF), the City of Vienna, and Berndorf Privatstiftung through Project MA16-005.

\bibliographystyle{plain}
\bibliography{References}

\end{document}